\documentclass[final,leqno]{amsart}
\usepackage{amsmath,amssymb,amsfonts,latexsym,stmaryrd}
\usepackage{graphicx,float,epsfig} 
\usepackage{mathrsfs} 
\usepackage{float} 
\usepackage{color}
\usepackage{setspace} 

\DeclareGraphicsExtensions{ .eps,.ps} \usepackage{subfig}
\usepackage{multirow}
\usepackage{url}
\usepackage{diagbox}
\usepackage[linesnumbered,algoruled,boxed]{algorithm2e}
\usepackage[colorlinks=false, linkcolor=blue,  citecolor=OliveGreen,  pdfstartview={}{}]{hyperref} 
\usepackage[dvipsnames]{xcolor}

\usepackage{svg}
\usepackage{amsmath}

\def\O{\Omega}

\newtheorem{remark}{Remark}[section]
\newtheorem{lemma}{Lemma}[section]
\newtheorem{theorem}{Theorem}[section]
\newtheorem{prop}{Proposition}[section]

\newcommand{\jump}[1]{\left\llbracket #1 \right\rrbracket}
\newcommand{\jumpp}[1]{\llbracket #1 \rrbracket}

\usepackage{booktabs}
\usepackage{float}

\newcommand\bu{\boldsymbol{u}}
\newcommand\bv{\boldsymbol{v}}
\newcommand\bw{\boldsymbol{w}}

\newcommand\bn{\boldsymbol{n}}

\newcommand\curl{\textbf{\text{curl}\,}}

\newcommand\bT{\boldsymbol{T}}

\newcommand\0{\mathbf{0}}

\newcommand\bxi{\boldsymbol{\xi}}

\def\CT{{\mathcal T}}

\newcommand{\dd}{\texttt{d}}

\newcommand\bsig{\boldsymbol{\sigma}}
\newcommand\btau{\boldsymbol{\tau}}
\newcommand\bPi{\boldsymbol{\Pi}}

\newcommand\R{\mathbb{R}}

\renewcommand\H{\mathrm{H}}
\renewcommand\L{\mathrm{L}}

\renewcommand\O{\Omega}

\newcommand\bdiv{\mathop{\mathbf{div}}\nolimits}

\renewcommand\div{\mathop{\mathrm{div}}\nolimits}

\newcommand\tr{\mathop{\mathrm{tr}}\nolimits}

\newcommand\LO{\L^2(\O)}

\newcommand\err{\texttt{err}}
\newcommand\eff{\texttt{eff}}

\newcommand\ds{\displaystyle}

\newcommand{\vertiii}[1]{{\left\vert\kern-0.25ex\left\vert\kern-0.25ex\left\vert #1 
    \right\vert\kern-0.25ex\right\vert\kern-0.25ex\right\vert}}

\begin{document}

\title[AFEM for a mixed formulation of the Stokes eigenproblem]
{A posteriori analysis for a mixed formulation of the Stokes spectral problem}


\author{Felipe Lepe}
\address{GIMNAP-Departamento de Matem\'atica, Universidad del B\'io - B\'io, Casilla 5-C, Concepci\'on, Chile.}
\email{flepe@ubiobio.cl}
\thanks{The first author was partially supported by
ANID-Chile through FONDECYT project 11200529.\\
The second author was partially supported by the National
Agency for Research and Development, ANID-Chile through project Anillo of
Computational Mathematics for Desalination Processes ACT210087, FONDECYT Postdoctorado project 3230302, and by project Centro de Modelamiento Matemático (CMM), FB210005, BASAL funds for centers of excellence.}


\author{Jesus Vellojin}
\address{GIMNAP-Departamento de Matem\'atica, Universidad del B\'io - B\'io, Casilla 5-C, Concepci\'on, Chile.}
\email{jvellojin@ubiobio.cl}


	

\begin{abstract}
In two and three dimensions, we design and analyze a posteriori error estimators for the mixed Stokes eigenvalue problem. The unknowns on this  mixed formulation are the pseudotress, velocity and pressure. With a lowest order mixed finite element scheme, together with  a postprocressing technique, we prove that the proposed estimator is reliable and efficient.  We illustrate the results with several numerical tests in two and three dimensions  in order to assess the performance of the estimator.
\end{abstract}

\maketitle

\section{Introduction}
\label{sec:intro}
Adaptive refinement strategies are an important subject of study in the numerical analysis 
of partial differential equations, since it is important to analyze if the proposed numerical schemes are
able to recover the optimal order of convergence when the solutions are not smooth enough. In the context of elliptic load problems, adaptive strategies have important applications and results in continuum mechanics. In particular, for the source Stokes problem the literature is abundant and well developed where different methods and formulations have been studied,  as in for instance
\cite{MR3556402,MR3892359, MR2293249,MR2970405,MR2594823,MR3568154,MR2871298} and the references therein.

For
eigenvalue problems, this is an ongoing topic where the  a posteriori analysis is currently being  developed for different problems, see \cite{MR4279087, MR3712172, MR3918688, MR3047040, LRV_vorticity, https://doi.org/10.48550/arxiv.2201.03658, MR2473688,  MR4050542, MR3715326}, just for mention some of the most recent. The reader can also see the references on  this articles for  further discussion.

Our interest is to continue with our research program related to a posteriori estimators for mixed eigenvalue problems that we begin in \cite{https://doi.org/10.48550/arxiv.2201.03658}. More precisely, the  design of  a posteriori error estimates for a mixed formulation of the Stokes spectral problem, recently introduced in   
\cite{MR4430561}, where a rigorous a priori analysis is performed for two families of mixed finite elements in two and three dimensions.  This formulation is able to consider the pressure and avoid it, leading to equivalent problems in the continuous level. Despite to this fact, the discrete formulations, the one containing the pressure as unknown and the one that avoids it, are not equivalent and hence, different discrete eigenvalue problems must be analyzed. However, for the a posteriori analysis, we will show that it  is possible to develop an a posteriori estimator for the pseudostress-pressure-velocity formulation and focus only on the analysis for this estimator, since in its definition, an a posteriori estimator for the pseudostress-velocity is contained, implying that all the analysis related to efficiency and reliability can be performed for both discrete formulations simultaneously.

From the above, and in order to complete the study of the pseudostress-pressure-velocity formulation for the Stokes eigenproblem, we propose a residual-based a posteriori error estimator. The analysis is performed for eigenvalues with simple multiplicity and its associated eigenfunctions. Using  a superconvergence result, we are able to control the high order terms that naturally appear when this kind of analysis is developed. The a posteriori estimator is constructed by means of lowest order Raviart-Thomas (RT) elements, suitable defined for  tensorial fields, which are considered to approximate the pseudotress tensor, whereas the velocity and pressure are approximated with piecewise linear functions. This is not the only alternative that we can consider as a numerical scheme for this formulation as is stated in \cite{MR4430561}, where Brezzi-Douglas-Marini (BDM) elements can be considered as an alternative to approximate the pseudostress.  However and for simplicity, the analysis is carried only with Raviart-Thomas elements, whereas in the numerical tests we do consider  the BDM family in order to observe the performance of the  adaptive algorithm  with this family  of finite elements. In addition, the mathematical and numerical analysis proposed in this study considers homogeneous Dirichlet  boundary conditions. However,  mixed boundary conditions can be also considered, and the analysis can be performed with minor modifications respect to the present contribution.

The paper is organized as follows: In section \ref{sec:model} we present the Stokes eigenvalue problem and the mixed formulation in consideration.  Also we summarize some necessary results to perform the analysis.   Section \ref{sec:fem}  is devoted to present the mixed finite element discretization of the Stokes eigenvalue problem. More precisely, we present the the lowest order Raviart-Thomas elements and its approximation properties, correctly adapted for the tensorial framework of the formulation. The core of our paper is  section \ref{sec:apost}, where we introduce the a posteriori error estimators for the full and reduced eigenvalue problems, the technical results needed to perform the analysis, and the results that establish that the error and the estimator are equivalent. Finally, in Section \ref{sec:numerics} we report numerical tests to assess the performance of the proposed adaptive scheme in two and three dimensions, proving  experimentally the efficiency and reliability of the  a posteriori estimators.

\subsection{Notations and preliminaries}
The following are some of the notations that will be used in this work. Given $n\in\{ 2,3\}$, we denote by $\mathbb{R}^{n}$ and $\mathbb{R}^{n\times n}$ the space of vectors and tensors of order $n$ with entries in $\mathbb{R}$, respectively. The symbol $\mathbb{I}$ represents the indentity matrix of $\mathbb{R}^{n\times n}$. Given any $\boldsymbol{\tau}:=(\tau_{ij})$ and $\boldsymbol{\sigma}:=(\sigma_{ij})\in \mathbb{R}^{n\times n}$, we write
\[
	\boldsymbol{\tau}^{\texttt{t}}:=(\tau_{ji}), \quad \tr(\boldsymbol{\tau}):=\sum_{i=1}^{n}\tau_{ii}, \quad  \boldsymbol{\tau:\sigma}:=\sum_{i,j=1}^{n} \tau_{ij}\,\sigma_{ij},\quad \mbox{and}  \quad\btau^\texttt{d}:=\btau- \frac{1}{n}\tr(\btau)\mathbb{I}
\]
to refer to the transpose, the trace, the tensorial product between $\boldsymbol{\tau}$ and $\boldsymbol{\sigma}$, and the deviatoric tensor of $\btau$, respectively. 

For $s\geq 0$, we denote as $\| \cdot \|_{s, \O}$ the norm of the Sobolev space $\H^{s}(\O)$, $[\H^{s}(\O)]^n$ or
$\mathbb{H}^s(\O):=[\H^{s}(\O)]^{n\times n}$ with $n\in\{2,3\}$  for scalar, vector,  and tensorial fields, respectively, with the convention $\H^0(\O):=\LO$, $[\H^0(\O)]^n:=[\LO]^n$, and $\mathbb{H}^0(\O):=\mathbb{L}^2(\O)$. Furthermore, with $\div$ denoting the usual divergence operator, we define the Hilbert space
\[
	\H(\div, \O):=\{ \boldsymbol{\tau} \in \L^{2}(\O) \,:\, \div(\boldsymbol{\tau}) \in \L^{2}(\O) \},
\] 
whose norm is given by $ \| \boldsymbol{\tau} \|_{\div, \O}^{2}:= \|\boldsymbol{\tau} \|_{0,\O}^{2} + \|\div(\boldsymbol{\tau}) \|_{0,\O}^{2} $. The space of matrix valued functions whose rows belong to $\H(\bdiv, \O)$ will be denoted $\mathbb{H}(\bdiv, \O)$ where $\bdiv$ stands for the action of $\div$ along each row of a tensor. Also, we introduce the space
$$\mathbb{H}(\curl, \Omega):=\{\bw\in\mathbb{L}^2(\O):\,\curl\bw\in\mathbb{L}^2(\O)\},$$
which is endowed with its natural norm.

Finally,  the relation $\texttt{a} \lesssim \texttt{b}$ indicates that $\texttt{a} \leq C \texttt{b}$, with a positive constant $C$ which is independent of $\texttt{a}$,  $\texttt{b}$ and the mesh size $h$, which will be introduced in Section \ref{sec:fem}. Similarly, we define $a\gtrsim b$ to denote $a\geq Cb$, with $C$ as above.

\section{The Stokes spectral problem}
\label{sec:model}
Introducing the pseudotress tensor $\bsig:=2\mu\nabla\bu-p\mathbb{I}$, the Stokes eigenvalue problem 
of our interest is the following:
\begin{equation}
\label{def:stokes_eigen}
\left\{
\begin{array}{rccc}
\bdiv\bsig&=&-\lambda\bu&\quad\text{in}\,\O\\
\bsig-2\mu\nabla\bu+p\mathbb{I}&=&\boldsymbol{0}&\quad\text{in}\,\O\\
\bdiv\bu&=&0&\quad\text{in}\,\O \\
\bu&=&\boldsymbol{0}&\quad\text{on}\,\partial\O,
\end{array}
\right.
\end{equation}
where $\mu$ is the kinematic viscosity and $\bdiv$ must be understood as the divergence of any tensor applied 
along on each row. As is commented in \cite{MR2594823}, the pressure and the pseudostress tensor are related
through the following identity
$
\ds p=-\tr(\bsig)/n
$
in $\O.$ This identity holds since $\tr(\nabla\bu)=\div\bu=0$. Hence, problem \eqref{def:stokes_eigen} can be rewritten as the following system:
\begin{equation}
\label{def:stokes_eigen2}
\left\{
\begin{array}{rccc}
\bdiv\bsig&=&-\lambda\bu&\quad\text{in}\,\O\\
\bsig-2\mu\nabla\bu+p\mathbb{I}&=&\boldsymbol{0}&\quad\text{in}\,\O\\
\displaystyle p+\frac{1}{n}\tr(\bsig)&=&0&\quad\text{in}\,\O \\
\bu&=&\boldsymbol{0}&\quad\text{on}\,\partial\O.
\end{array}
\right.
\end{equation}

A variational formulation for \eqref{def:stokes_eigen2} in terms of the deviatoric tensors $\bsig^\texttt{d}$ and $\btau^\texttt{d}$ is (see for example \cite{MR3860570}):  Find $\lambda\in\mathbb{R}$ and the triplet $((\bsig,p),\bu)\in\mathbb{H}(\bdiv,\O)\times \L^2(\O)\times[\L^2(\O)]^n$ such that 
\begin{align}
	&\frac{1}{2\mu}\int_{\O}\bsig^{\dd}:\btau^{\dd}+\frac{n}{2\mu}\int_{\O}\left(p+\frac{1}{n}\tr(\bsig)\right)\left(q+\frac{1}{n}\tr(\btau)\right)+\int_{\O}\bu \cdot \bdiv\btau=0,\label{eq:weak-principal1}\\
	&\int_{\O}\bv \cdot \bdiv\bsig=-\lambda\int_{\O}\bu\cdot\bv,\label{eq:weak-principal2}
\end{align}
for all $((\btau,q),\bv)\in\mathbb{H}(\bdiv,\O)\times \L^2(\O)\times[\L^2(\O)]^n$. However, the solution for this problem is not unique if homogeneous Dirichlet conditions on the whole boundary are considered \cite[Lemma 2.1]{MR2594823}. This is circumvented by requiring that $\bsig\in\mathbb{H}_0$, where the space $\mathbb{H}_0$ is given in the decomposition $\mathbb{H}(\bdiv,\O)=\mathbb{H}_0\oplus\mathbb{R}\mathbb{I}$, with
\begin{equation*}
	\mathbb{H}_0:=\left\{\btau\in\mathbb{H}(\bdiv,\O)\,:\,\int_{\O}\tr(\btau)=0\right\}.
\end{equation*}   
Here, and in the rest of the paper, we will assume that $\bsig\in\mathbb{H}_0$.

For simplicity, we define $\mathbb{H}:=\mathbb{H}_0\times \L^2(\O)$. Hence, following \cite[Lemma 2.2]{MR2594823} we have that $\bsig\in \mathbb{H}_0$ is solution of \eqref{eq:weak-principal1}--\eqref{eq:weak-principal2}, which is restated as: Find $\lambda\in\mathbb{R}$ and the triplet $((\boldsymbol{0},0),\boldsymbol{0})\neq ((\bsig,p),\bu)\in\mathbb{H}\times \L^2(\O)\times[\L^2(\O)]^n$ such that
\begin{align}
\label{eq1}a((\bsig,p),(\btau,q))+b(\btau,\bu)&=0\,\,\,\,\quad\quad\quad\forall (\btau,q)\in\mathbb{H},\\
\label{eq2}b(\bsig,\bv)&=-\lambda(\bu,\bv)\quad\forall\bv\in  \mathbf{Q},\,
\end{align}
where  $\mathbf{Q}:=[\L^2(\O)]^n$ and the bilinear forms $a:\mathbb{H}\times \mathbb{H}\rightarrow\mathbb{R}$ and $b:\mathbb{H}\times \mathbf{Q}\rightarrow\mathbb{R}$ are defined by
\begin{equation*}
a((\bxi,r),(\btau,q)):=\frac{1}{2\mu}\int_{\O}\bsig^{\dd}:\btau^{\dd}+\frac{n}{2\mu}\int_{\O}\left(r+\frac{1}{n}\tr(\bxi)\right)\left(q+\frac{1}{n}\tr(\btau)\right),
\end{equation*}
and 
\begin{equation*}
b(\bxi,\bv):=\int_{\O}\bv\cdot\bdiv\bxi.
\end{equation*}

\begin{remark}
	In \cite{MR4430561} is considered a reduced formulation where the pressure $p$ can be eliminated. For instance, we can consider the problem: Find $\lambda\in\mathbb{R}$ and $(\0,\0) \neq(\bsig,\bu)\in \mathbb{H}_0\times \mathbf{Q}$ such that
	\begin{align}
		\label{eq1_reduced}a_0(\bsig,\btau)+b(\btau,\bu)&=0\,\,\,\,\quad\quad\quad\forall \btau\in \mathbb{H}_0,\\
		\label{eq2_reduced}b(\bsig,\bv)&=-\lambda(\bu,\bv)\quad\forall\bv\in \mathbf{Q},
	\end{align}
	where $a_0: \mathbb{H}_0\times \mathbb{H}_0\rightarrow\mathbb{R}$ is a bounded bilinear form defined by
	\begin{equation*}
		\displaystyle a_0(\bxi,\btau):=\frac{1}{2\mu}\int_{\O}\bxi^{\texttt{d}}:\btau^{\texttt{d}}\quad\forall (\bxi,\btau)\in \mathbb{H}_0\times \mathbb{H}_0.
	\end{equation*}
	The analysis can be performed with this reduced formulation, but for this paper, we are interested on \eqref{eq1}--\eqref{eq2}. Moreover, at discrete level, the a posteriori estimator for the FEM discretization of  \eqref{eq1}--\eqref{eq2} contains terms of the reduced problem. Hence, we are in some way considering both problems at the same time.
\end{remark}
We recall that \eqref{eq1}--\eqref{eq2} and \eqref{eq1_reduced}--\eqref{eq2_reduced}  are equivalent, however their finite element counterparts are not  (see \cite{MR4430561} for instance).

From  \cite{MR975121,MR1600081} we have the following regularity result for the Stokes spectral problem.
\begin{theorem}
\label{th:reg_velocity}
There exists $s>0$ such that $\bu\in [\H^{1+s}(\Omega)]^n$ and $p\in \H^s(\Omega)$.
\end{theorem}
The well posedness of \eqref{eq1}--\eqref{eq2} implies the existence of an operator $\mathcal{A}:\mathbb{H}\times\mathbf{Q}\rightarrow(\mathbb{H}\times\mathbf{Q})'$, induced by the left-hand side of \eqref{eq1}--\eqref{eq2}, which is an isomorphism that satisfies $\|\mathcal{A}((\btau,q),\bv)\|_{(\mathbb{H}\times\mathbf{Q})'}\gtrsim\|((\btau,q),\boldsymbol{v})\|_{\mathbb{H}\times\mathbf{Q}}$, for all $((\btau,q),\boldsymbol{v})\in\mathbb{H}\times\mathbf{Q}$, that  is equivalent to the following inf-sup condition
\begin{equation}
\label{eq:complete_infsup}
\|((\boldsymbol{\tau},q), \bv)\|_{\mathbb{H}\times \mathbf{Q}}\,\lesssim \displaystyle\sup_{\underset{((\boldsymbol{\xi},r),\boldsymbol{w})\neq \boldsymbol{0}}{((\boldsymbol{\xi},r),\boldsymbol{w})\in \mathbb{H}\times\mathbf{Q}}}\frac{a((\boldsymbol{\tau},q),(\boldsymbol{\xi},r)) + b(\boldsymbol{\xi},\bv) + b(\btau,\boldsymbol{w})}{\|((\boldsymbol{\xi},r),\boldsymbol{w})\|_{\mathbb{H}\times \mathbf{Q}}}.
\end{equation}


\section{The mixed element method}
\label{sec:fem}
The present section deals with the finite element approximation for the eigenvalue 
problem \eqref{eq1}--\eqref{eq2}.  For $\Omega\subset\mathbb{R}^\texttt{d}$, with $\texttt{d}\in\{2,3\}$,  let $\mathcal{T}_h$ be a shape regular family of meshes which subdivide the domain $\bar \Omega$ into  triangles/tetrahedra that we denote by $T$. Let $h_T$ be the diameter of a triangle/tetrahedron $T$ of the triangulation and let us define $h:=\max\{h_T\,:\, T\in \CT_h\}$. Given an integer $\ell\geq 0$ and a subset $D$ of $\mathbb{R}^n$, we denote by $\mathbb{P}_\ell(D)$ the space of polynomials of degree at most $\ell$ defined in $D$. With these ingredients at hand, for $\ell=0$ we define the local Raviart-Thomas space of the lowest order as follows  (see \cite{MR3097958})
 \begin{equation*}
 \mathbf{RT}_0(T)=[\mathbb{P}_0(T)]^n\oplus \mathbb{P}_0(T)\boldsymbol{x},
 \end{equation*}
 where $\boldsymbol{x}\in\mathbb{R}^n$. With this local space at hand, we define the global Raviart-Thomas space, which we denote by $\mathbb{RT}_0(\CT_h)$, as follows
 \begin{equation*}
 \mathbb{RT}_0(\CT_h):=\{\btau\in\mathbb{H}(\bdiv,\O)\,:\,(\tau_{i1},\cdots,\tau_{in})^{\texttt{t}}\in\mathbf{RT}_0(T)\,\,\forall i\in\{1,\ldots,n\},\,\,\forall T\in\CT_h\}.
 \end{equation*}
Also  we introduce the global space of piecewise polynomials of degree $\leq k$ defined by
  \begin{equation*}
 	\mathbb{P}_k(\CT_h):=\{v\in\L^2(\O)\,:\, v|_T\in\mathbb{P}_k(T)\,\,\forall T\in\CT_h\}.
 \end{equation*}

With these discrete spaces at hand, we recall some approximation properties that hold for each of them (see \cite{MR2009375} for instance).  

Let $\bPi_h:\mathbb{H}^t (\O)\rightarrow \mathbb{RT}_0(\CT_h)$ be the Raviart-Thomas interpolation operator. For $t\in (0,1/2]$ and $\btau\in\mathbb{H}^t(\O)\cap\mathbb{H}(\bdiv;\O)$ the following error estimate holds true
\begin{equation} \label{daniel1}
	\|\btau-\bPi_h\btau\|_{0,\O}\lesssim h^t \big(\|\btau\|_{t,\O}+\|\bdiv\btau\|_{0,\O}\big).
\end{equation}
Also, for $\btau\in\mathbb{H}^t(\O)$ with $t>1/2$, there holds
\begin{equation}\label{daniel2}
	\|\btau-\bPi_h\btau\|_{0,\O}\lesssim h^{\min\{t,1\}} |\btau|_{t,\O}.
\end{equation} 

Let $\mathcal{P}_h:[\L^2(\O)]^n\rightarrow[\mathbb{P}_0(\mathcal{T}_h)]^n$ be the $\L^2(\O)$-orthogonal projector, which satisfies the following commuting diagram:
\begin{equation}
	\label{eq:commutative}
	\bdiv(\bPi_h\btau)=\mathcal{P}_h(\bdiv\btau),
\end{equation}
and, if $\bv\in\H^t (\O)^{n}$ with $t>0$, $\mathcal{P}_h$ it also satisfies
\begin{equation}\label{daniel3}
	\|\bv-\mathcal{P}_h\bv\|_{0,\O}\lesssim h^{\min\{t,1 \}} |\bv|_{t,\O}.
\end{equation}

Finally, for each $\btau\in\mathbb{H}^t(\O)$ such that $\bdiv\btau\,\in[\H^t (\O)]^{n}$, there holds
\begin{equation} \label{daniel4}
	\|\bdiv(\btau-\bPi_h\btau)\|_{0,\O}\lesssim h^{\min\{t,1\}} |\bdiv\btau|_{t,\O}.
\end{equation} 

It is worth noting that all the following analysis is also valid if the Brezzi-Douglas-Marini family, namely $\mathbb{BDM}$, is used (see \cite{MR4430561} for details and Section \ref{sec:numerics} below).

To end this section, we define
\begin{equation*}
	\mathbb{H}_{0,h}:=\left\{ \btau_h\in\mathbb{RT}_0(\CT_h)\,\,:\,\,\int_{\O}\tr(\btau_h)=0  \right\},
\end{equation*}
and also define $Q_h:=\mathbb{P}_0(\CT_h)$,  $\mathbf{Q}_h:=[\mathbb{P}_0(\CT_h)]^n$ and $\mathbb{H}_h:=\mathbb{H}_{0,h}\times Q_h $.

\subsection{The discrete eigenvalue problems}
With the discrete spaces defined above, we are in position to introduce the discretization of 
problem \eqref{eq1}--\eqref{eq2}: Find $\lambda_h\in\mathbb{R}$ and $((\boldsymbol{0},0),\boldsymbol{0})\neq((\bsig_h, p_h),\bu_h)\in\mathbb{H}_h\times\mathbf{Q}_h$ such that
\begin{align}
	\label{eq1h}a((\bsig_h,p_h),(\btau_h,q_h))+b(\btau_h,\bu_h)&=0\,\,\,\,\quad\quad\quad\forall (\btau_h,q_h)\in\mathbb{H}_h,\\
	\label{eq2h}b(\bsig_h,\bv_h)&=-\lambda_h(\bu_h,\bv_h)\quad\forall\bv_h\in \mathbf{Q}_h.
\end{align}

Similarly as in the continuous case, it is possible to consider a reduced formulation for the discrete eigenvalue problem which  reads as follows: Find $\lambda_h\in\mathbb{R}$ and $(\boldsymbol{0},\0)\neq (\bsig_h,\bu_h)\in \mathbb{H}_{0,h}\times \mathbf{Q}_h$ such that
\begin{align}
	\label{reduced_disc_source1}a_0(\bsig_h,\btau_h)+b(\btau_h,\bu_h)&=0\,\,\,\,\quad\quad\quad\forall \btau_h\in \mathbb{H}_{0,h},\\
	\label{reduced_disc_source2}b(\bsig_h,\bv_h)&=-\lambda(\bu_h,\bv_h)\quad\forall\bv_h\in \mathbf{Q}_h.
\end{align}

A priori error estimates for problems \eqref{eq1}--\eqref{eq2} and \eqref{eq1h}--\eqref{eq2h} are derived from \cite[Theorems 4.4 and 4.5]{MR4430561}.
\begin{lemma}
\label{lema:apriorie}
Let $(\lambda, ((\boldsymbol{\sigma},p),\bu))$ be a solution of  \eqref{eq1}--\eqref{eq2} with $\|\bu\|_{0,\O}=1$, and let $(\lambda_h, ((\boldsymbol{\sigma}_h,p_h),\bu_h))$ be its finite element approximation given as the solution to \eqref{eq1h}--\eqref{eq2h} with $\|\bu_h\|_{0,\O}=1$. Then
\begin{equation*}
\|\boldsymbol{\sigma}-\boldsymbol{\sigma}_h\|_{0,\O}+\|p-p_h\|_{0,\O}+\|\bu-\bu_h\|_{0,\O}\lesssim  h^{s},
\end{equation*}
and
\begin{equation*}
|\lambda-\lambda_h|\lesssim \|\boldsymbol{\sigma}-\boldsymbol{\sigma}_h\|_{0,\O}^2+\|p-p_h\|_{0,\O}^2+\|\bu-\bu_h\|_{0,\O}^2,
\end{equation*}
where, the hidden constants are  independent of $h$, and $s>0$ as in Theorem \ref{th:reg_velocity}.
\end{lemma}

We end this section with the following technical result, which  states that for $h$ small enough, except 
for $\lambda_h$, the rest of the eigenvalues of \eqref{eq1h}--\eqref{eq2h}  are well separated from $\lambda$ (see  \cite{MR3647956}). 

\begin{prop}\label{separa_eig}
Let us enumerate the eigenvalues of  systems \eqref{eq1h}--\eqref{eq2h}  and  \eqref{eq1}--\eqref{eq2} in increasing order as follows: $0<\lambda_1\leq\cdots\lambda_i\leq\cdots$ and 
$0<\lambda_{h,1}\leq\cdots\lambda_{h,i}\leq\cdots$. Let us assume  that $\lambda_J$ is a simple eigenvalue of \eqref{eq1h}--\eqref{eq2h} . Then, there exists $h_0>0$ such that
\begin{equation*}
|\lambda_J-\lambda_{h,i}|\geq\frac{1}{2}\min_{j\neq J}|\lambda_j-\lambda_J|\quad\forall i\leq \dim\mathbb{H}_h,\,\,i\neq J,\quad \forall h<h_0.
\end{equation*}
\end{prop}

\section{A posteriori error analysis}
\label{sec:apost}
The aim of the following section is  to design and analyze an a posteriori error estimator for our 
mixed formulation of the Stokes spectral problem in two and three dimensions.
This implies not only the analysis of the efficiency and reliability bound, but also the control of the high order terms that naturally appear on this analysis. With this goal in mind, we will adapt the results of \cite{MR3047040} in order to obtain a superconvergence result that is needed to precisely handle the high order terms.
\subsection{Properties of the mesh}
Let us set some definitions. For $T\in\mathcal{T}_h$, let $\mathcal{E}(T)$ be the set of its faces/edges, and let $\mathcal{E}_h$ be the set of all
the faces/edges of the mesh $\mathcal{T}_h$. With these definitions at hand, we write $\mathcal{E}_h=\mathcal{E}_h(\O)\cup\mathcal{E}_h(\partial\O)$, where
\begin{equation*}
\mathcal{E}_h(\O):=\{e\in\mathcal{E}_h\,:\,e\subseteq\O\}\quad\text{and}\quad\mathcal{E}_h(\O):=\{e\in\mathcal{E}_h\,:\, e\subseteq\partial\O\}.
\end{equation*}

For each face/edge $e\in\mathcal{E}_h$ we fix a unit normal vector $\bn_e$ to $e$. Moreover, given $\btau\in\mathbb{H}(\curl; \Omega)$ and $e\in\mathcal{E}_h(\O)$, we let $\jumpp{\btau\times\bn_e}$ be the corresponding jump of the tangential traces across $e$, defined by
\begin{equation*}
\jumpp{\btau\times\bn_e}:=(\btau|_T-\btau|_{T'})\big|_e\times\bn_e,
\end{equation*} 
where $T$ and $T'$ are two elements of the mesh with common edge $e$.
For two dimensions, the tangential traces across $e$ is defined by
\begin{equation*}
	\jumpp{\btau\times\bn_e}:=\left[(\btau|_T-\btau|_{T'})\big|_e\right]\mathbf{t}_e,
	\end{equation*}
where $\mathbf{t}_e:=(-n_2,n_1)$ is the corresponding tangential vector for the facet normal $\bn_e=(n_1,n_2)$.
 
\subsection{Postprocessing}

Let us define the following space
\begin{equation*}
\mathrm{Y}_h:=\{\bv\in [\H^1(\Omega)]^n\,:\,\bv\in [P_1(T)]^n, \quad\forall T\in\mathcal{T}_h\}.
\end{equation*}
For each vertex $z$ of the elements in $\mathcal{T}_h$, we define the patch
\begin{equation*}
\omega_z:=\bigcup_{z\in T\in\mathcal{T}_h} T.
\end{equation*}

We introduce  the postprocessing operator $\Theta_h:\mathbf{Q}\rightarrow \mathrm{Y}_h$. The motivation of this operator is to  fit a piecewise linear function in the average sense, for any $\bv\in\mathbf{Q}$ at the degrees of freedom of element integrations in the following way
\begin{equation*}
\displaystyle\Theta_h\bv(z):=\sum_{T\in\omega_z}\frac{\displaystyle\int_T\bv\,dx}{|\omega_z|},
\end{equation*}
where $|\omega_z|$ denotes the facet measure (area in 2D, volume in 3D) of the patch. Let us precise that $\Theta_h\boldsymbol{v}(z)$ is defined on a vertex $z\in\omega_z$ that coincides with one of the degrees of freedom needed to 
	define a function of $\mathrm{Y}_h$. Then, $\Theta_h\boldsymbol{v}(z)$ is computed by averaging the sum of the integrals of $\boldsymbol{v}\in\mathbf{Q}$ over all the elements sharing this vertex.

Let us recall the properties that $\Theta_h$ satisfy (see \cite[Lemma 3.2, Theorem 3.3]{MR3047040}).
\begin{lemma}[Properties of the postprocessing operator]
\label{postprocessing} 
The operator $\Theta_h$ defined above satisfies the following:
\begin{enumerate}
\item For $\bu\in [\H^{1+s}(\Omega)]^n$ with $s>0$ as in Theorem \ref{th:reg_velocity} and $T\in\mathcal{T}_h$, there holds $\|\Theta_h\bu-\bu\|_{0,T}\lesssim h_T^{1+s}\|\bu\|_{1+s,\omega_T}$;
\item $\Theta_h\mathcal{P}_h\bv=\Theta_h\bv$;
\item $\|\Theta_h\bv\|_{\mathbf{Q}}\lesssim\|\bv\|_{\mathbf{Q}}$ for all $\bv\in\mathbf{Q}$.
\end{enumerate}
\end{lemma}

Also, we have the following approximation result.
\begin{lemma}
\label{lmm:rodolfo11}
Let $(\lambda,((\boldsymbol{\sigma},p), \bu))$ and $(\lambda_{h}, ((\boldsymbol{\sigma}_{h},p_h),\bu_{h}))$ be  solutions of Problems \eqref{eq1}--\eqref{eq2} and \eqref{eq1h}--\eqref{eq2h}, respectively, with $\|\bu\|_{0,\O}=\|\bu_{h}\|_{0,\O}=1$. Then, there holds
\begin{equation*}
\|\mathcal{P}_h\bu-\bu_{h}\|_{0,\O}\lesssim h^{s}\left(\|\boldsymbol{\sigma}-\boldsymbol{\sigma}_h\|_{0,\O}+\|p-p_h\|_{0,\O}+\|\bu-\bu_h\|_{0,\O}\right),
\end{equation*}
where $s>0$ and the hidden constant is independent of $h$.
\end{lemma}

With Lemmas \ref{lmm:rodolfo11} and \ref{postprocessing} at hand,  have the following superconvergence result for $\Theta_h$ (see \cite[Theorem 3.3]{ MR3047040} for the proof).

\begin{lemma}[superconvergence]
\label{lmm:super}
For $h$ small enough, there holds
\begin{equation*}
\|\Theta_h\bu_h-\bu\|_{0,\O}\lesssim h^{s}\left(\|\boldsymbol{\sigma}-\boldsymbol{\sigma}_h\|_{0,\O}+\|p-p_h\|_{0,\O}+\|\bu-\bu_h\|_{0,\O}\right)+\|\Theta_h\bu-\bu\|_{0,\O},
\end{equation*}
where the hidden constant is independent of $h$.
\end{lemma}

The following auxiliary results, available in \cite{MR3453481}, are necessary in our forthcoming analysis.
\subsection{Technical tools} To perform the analysis, we recall two key  properties that are needed. 
	
First, let us consider the operator  $I_{h}:\H^{1}(\O)\rightarrow \mathscr{C}_I$, where $
	\mathscr{C}_I:=\{v\in C(\bar{\Omega}): v|_{T}\in \mathrm{P}_{1}(T) \;\ \forall T\in\CT_{h}\}
	$ is the Cl\'ement interpolant of degree $k=1$ (see \cite[Chapter 2.]{koelink2006partial}). Similarly, we define $\boldsymbol{I}_h:[\H^1(\Omega)]^n\rightarrow [\mathscr{C}_I]^n=\mathrm{Y}_h$ as the vectorial version of $I_h$.

We now establish the following lemma, which states the local approximation properties of $I_{h}$.
\begin{lemma}
\label{I:clemont}
For all $v\in \H^{1}(\O)$ there holds
\begin{equation*}
\|v-I_{h}v\|_{0,T}\lesssim  h_{T}\|v\|_{1,\omega_{T}}\quad \forall T\in\CT_{h},
\end{equation*}
and
\begin{equation*}
\|v-I_{h}v\|_{0,e}\lesssim h_{e}^{1/2}\|v\|_{1,\omega_{e}}\quad \forall e\in \mathcal{E}_h,
\end{equation*}
where $\omega_{T}:=\{T'\in\CT_{h}: T' \text { and } T \text{ share a facet}\}$,  $\omega_{e}:=\{T'\in\CT_{h}: e\in \mathcal{E}_{T'}\}$, and the hidden constants are independent of $h$.
\end{lemma}
Secondly, the  following Helmoltz decomposition holds (see \cite[Lemma 4.3]{MR3453481}).
\begin{lemma}
\label{lm:helm}
For each $\btau\in \mathbb{H}(\bdiv,\O)$ there exist $\boldsymbol{z}\in [\H^{2}(\O)]^{n}$ and $\boldsymbol{\chi}\in \mathbb{H}^1(\Omega)$ such that
\begin{equation*}
\btau=\nabla \boldsymbol{z}+\mathbf{curl}{\boldsymbol{\chi}}\quad \text{ in }\O\quad \text{ and }\quad \|\boldsymbol{z}\|_{2,\O}+\|\boldsymbol{\chi}\|_{1,\O}\lesssim \|\btau\|_{\bdiv,\O},
\end{equation*}
where the hidden constant independent of all the foregoing variables.
\end{lemma}

\subsection{The local and global error indicators} 
Now we present the local indicators for our problem. We need to remark that the reduced and full problems are not equivalent, and hence, each formulation have a particular local indicator. Let us present the indicator for the 
reduced discrete eigenvalue problem
\begin{multline}
	\label{eq:local_reduced}
	\theta_T^2:=\|\Theta_h\bu_h-\bu_h\|_{0,T}^2+h_T^2\left\|\curl\left\{\frac{1}{2\mu}\bsig_h^d\right\}\right\|_{0,T}^2+h_T^2\left\|\nabla\bu_h-\frac{1}{2\mu}\bsig_h^d\right\|_{0,T}^2\\
	\small+\sum_{e\in\mathcal{E}(T)\cap\mathcal{E}_h(\O)}h_e\left\|\jump{\frac{1}{2\mu}\bsig_h^d\times\bn_e}\right\|_{0,e}^2
	+
	\sum_{e\in\mathcal{E}(T)\cap\mathcal{E}_h(\partial\O)}h_e\left\|\frac{1}{2\mu}\bsig_h^d\times\bn_e\right\|_{0,e}^2,
\end{multline}
and the global estimator is defined by
\begin{equation}
\label{eq:global_est}
\theta:=\left\{ \sum_{T\in\mathcal{T}_h}\theta_{T}^2\right\}^{1/2}.
\end{equation}
Now, the local indicator for the complete formulation incorporates the contributions of the pressure, together with \eqref{eq:local_reduced} as follows
\begin{multline}
\label{eq:local_full}
\eta_T^2:=\theta_T^2+\left\|p_h+\frac{1}{n}\tr(\bsig_h)\right\|_{0,T}^2+h_T^2\left\|\curl\left[\left( p_h+\frac{1}{n}\tr(\bsig_h)\right)\mathbb{I}\right]\right\|_{0,T}^2\\
+\sum_{e\in\mathcal{E}(T)\cap\mathcal{E}_h(\O)}h_e\left\|\jump{\left[\left( p_h+\frac{1}{n}\tr(\bsig_h)\right)\mathbb{I}\right]\times\bn_e}\right\|_{0,e}^2\\
+\sum_{e\in\mathcal{E}(T)\cap\mathcal{E}_h(\partial\O)}h_e\left\|\left[\left( p_h+\frac{1}{n}\tr(\bsig_h)\right)\mathbb{I}\right]\times\bn_e\right\|_{0,e}^2,
\end{multline}
and hence, the global estimator for the complete problem is 
\begin{equation}
\label{eq:global_full}
\eta:=\left\{ \sum_{T\in\mathcal{T}_h}\eta_{T}^2\right\}^{1/2}.
\end{equation}

It is important to notice that for  $n=2$  the tangential traces are taken as $\bsig_h^{\dd}\mathbf{t}_e/(2\mu)$ and $\left( p_h+\tr(\bsig_h)/n\right)\mathbf{t}_e$ for the estimators \eqref{eq:local_reduced} and \eqref{eq:local_full}, respectively, where $\mathbf{t}$ is the corresponding   unit tangential vector along the edge $e$.

Now our task is to analyze the reliability and efficiency of \eqref{eq:global_est} and \eqref{eq:global_full}. Let us claim that our attention will be focused on the complete estimator \eqref{eq:global_full}, since for \eqref{eq:global_est} the computations are straightforward from the complete case.
\subsection{Reliability}
The goal of this section is to derive an upper bound for  \eqref{eq:global_full}. Let us begin with the following result.

\begin{lemma}
\label{lema_cota_s}
Let $(\lambda,((\bsig, p),\bu))\in\R\times\mathbb{H}\times \mathbf{Q}$ be the solution of \eqref{eq1}--\eqref{eq2} and let $(\lambda_{h} ,((\boldsymbol{\sigma}_h,p_h), \bu_h))\in\R\times\mathbb{H}_{h}\times\mathbf{Q}_h$ be its finite element approximation, given as the  solution of  \eqref{eq1h}--\eqref{eq2h}.  Then for all $\btau\in \mathbb{H}_{0}$, we have
\begin{multline}
\label{eq:error_bound1}
\|\boldsymbol{\sigma}-\boldsymbol{\sigma}_{h}\|_{\bdiv,\O}+\|p-p_h\|_{0,\O}+\|\bu-\bu_{h}\|_{0,\O}\\
\lesssim \displaystyle\sup_{\underset{(\btau,q)\neq(\boldsymbol{0},0)}{(\btau,q)\in\mathbb{H}}}\frac{-a((\boldsymbol{\sigma}_{h},p_h),(\btau,q))-b(\btau,\bu_h)}{\|(\btau,q)\|_{\bdiv,\O}}\\
+\underbrace{|\lambda_{h}-\lambda |+\|\bu-\Theta_h\bu_{h}\|_{0,\O}}_{\text{h.o.t}}+\|\Theta_h\bu_{h}-\bu_{h}\|_{0,\O}.
\end{multline}
\end{lemma}
\begin{proof}
Applying the inf-sup condition \eqref{eq:complete_infsup} on the errors $\boldsymbol{\sigma}-\boldsymbol{\sigma}_h$, $p-p_h$ and $\bu-\bu_h$, together with \eqref{eq1},  we have
that 
\begin{multline*}
\|((\boldsymbol{\sigma}-\boldsymbol{\sigma}_{h},p-p_h),\bu-\bu_{h})\|_{\mathbb{H}\times \mathbf{Q}}\\
\lesssim \displaystyle\sup_{\underset{((\btau,q),\bv)\neq \boldsymbol{0}}{((\btau,q),\bv)\in \mathbb{H}\times\mathbf{Q}}}\frac{a((\boldsymbol{\sigma}-\boldsymbol{\sigma}_{h},p-p_h),(\btau,q))+b(\btau,\bu-\bu_h)+b(\boldsymbol{\sigma}-\boldsymbol{\sigma}_h,\bv)}{\|((\btau,q),\bv)\|_{\mathbb{H}\times \mathbf{Q}}}\\
\lesssim \displaystyle\sup_{\underset{(\btau,q)\neq(\boldsymbol{0},0)}{(\btau,q)\in\mathbb{H}}}\frac{-a((\boldsymbol{\sigma}_{h},p_h),(\btau,q))-b(\btau,\bu_h)}{\|\btau\|_{\bdiv,\O}+\|q\|_{0,\O}}+\displaystyle\sup_{\underset{\bv\neq\boldsymbol{0}}{\bv\in\mathbf{Q}}}\frac{b(\boldsymbol{\sigma}-\boldsymbol{\sigma}_h,\bv)}{\|\bv\|_{0,\O}},
\end{multline*}
Now, according to the definition of the bilinear form $b(\cdot,\cdot)$,  \eqref{eq2}, the relation $\bdiv\boldsymbol{\sigma}_{h}=-\lambda_h\bu_h$,  and the Cauchy–Schwarz inequality, we obtain
\begin{align*}
\displaystyle\sup_{\underset{\bv\neq\boldsymbol{0}}{\bv\in\mathbf{Q}}}\frac{b(\boldsymbol{\sigma}-\boldsymbol{\sigma}_h,\bv)}{\|\bv\|_{0,\O}}\leq |\lambda_h-\lambda|\|\bu_h\|_{0,\O}+|\lambda|\left(\|\bu-\Theta_h\bu_h\|_{0,\O}+\|\Theta_h\bu_h-\bu_h\|_{0,\O}\right).
\end{align*}
The proof is concluded by similar arguments of those presented on the proof of \cite[Lemma 4.5]{https://doi.org/10.48550/arxiv.2201.03658}.
\end{proof}

	We note that thanks to Lemmas \ref{lema:apriorie}, \ref{postprocessing} and \ref{lmm:super}, if $(\lambda, (\boldsymbol{\sigma},p),\bu)\in\mathbb{R}\times\mathbb{H}_0^s(\Omega)\times\H^{s}(\Omega)\times[\H^{1+s}(\Omega)]^n$ is the solution of \eqref{eq1}--\eqref{eq2} with $\|\bu\|_{0,\O}=1$ and \break $(\lambda_h, (\boldsymbol{\sigma}_h, p_h),\bu_h)\in\mathbb{R}\times\mathbb{H}_{h}\times\mathbf{Q}$ is the solution of \eqref{eq1h}--\eqref{eq2h}with $\|\bu_h\|_{0,\O}=1$, then, the high order term satisfy
		\begin{multline}
		\label{eq:hot}
			\mathrm{h.o.t}\lesssim h^s \left(\|\boldsymbol{\sigma}-\boldsymbol{\sigma}_{h}\|_{0,\O}+\|p-p_h\|_{0,\O}+\|\bu-\bu_{h}\|_{0,\O}\right)\\
			+\|\bu-\Theta_h\bu\|_{0,\O}
			\lesssim h^{2s},
		\end{multline}		
		where the hidden constant is independent of the mesh size and $s>0$ as in Theorem \ref{th:reg_velocity}.
	

To bound the supremum in \eqref{eq:error_bound1}, we proceed analogously as in  \cite{MR2594823,https://doi.org/10.48550/arxiv.2201.03658}. Indeed, considering the decomposition  $\btau=\nabla\boldsymbol{z}+\curl\boldsymbol{\chi}$ given by  Lemma \ref{lm:helm}, where $\btau\in\mathbb{H}_0$, we notice that is possible to define $\btau_h\in\mathbb{H}_{0,h}$ through  a discrete Helmholtz decomposition by
\begin{equation*}
\btau_{h}:=\bPi_h\left(\nabla\boldsymbol{z}\right)+\curl(\boldsymbol{\chi}_{h})-d_{h}\mathbb{I},
\end{equation*}
where 
$\boldsymbol{\chi}_{h}:=(\boldsymbol{\chi}_{1h}, \ldots,\boldsymbol{\chi}_{nh})^{\texttt{t}}$, with $\boldsymbol{\chi}_{ih}:=\boldsymbol{I}_{h}(\boldsymbol{\chi}_{i})$ for $i=\{1,...,n\}$, $\bPi_h$ is the Raviart-Thomas interpolation operator, satisfying \eqref{daniel1}-\eqref{daniel4} and the constant $d_{h}$ is chosen as
\begin{equation*}
d_{h}:=\dfrac{1}{n|\O|}\int_{\O}\tr(\btau_{h})=-\dfrac{1}{n|\O|}\int_{\O}\tr\left(\nabla\boldsymbol{z}-\bPi_h\left(\nabla\boldsymbol{z}\right)+\curl(\boldsymbol{\chi}-\boldsymbol{\chi}_{h})\right).
\end{equation*}
Following the arguments provided by \cite{LRV_vorticity,https://doi.org/10.48550/arxiv.2201.03658}
for $\boldsymbol{\sigma}_{h},\in\mathbb{H}_{0,h}$, setting  $\bxi=\btau-\btau_h$, and with the aid the Helmholtz decomposition, we  obtain the identity
\begin{multline*}
-\left[a((\boldsymbol{\sigma}_{h},p_h),(\btau,q))+b(\btau,\bu_h)\right]=-\left[a((\boldsymbol{\sigma}_{h},p_h),(\bxi,q))+b(\bxi,\bu_h)\right]\\=-a((\boldsymbol{\sigma}_{h},p_h),(\bxi,q)),
\end{multline*}
Hence, we have 
\begin{multline}
\label{eq:sup}
a((\bsig_h,p_h),(\btau,q))+b(\btau,\bu_h)=\underbrace{a((\bsig_h,p_h),(\nabla\boldsymbol{z}-\bPi_h(\nabla\boldsymbol{z}),q))}_{\mathbf{I}}\\
+\underbrace{a((\bsig_h,p_h),(\curl(\boldsymbol{\chi}-\boldsymbol{\chi}_h),q))}_{\mathbf{II}}.
\end{multline}
Now,  each contribution $\mathbf{I}$ and $\mathbf{II}$ is controlled using the arguments provided by \cite[Section 4.1]{MR2594823} and \cite[Lemmas 4.6 and 4.7]{https://doi.org/10.48550/arxiv.2201.03658}. Therefore,  we obtain that 
	\begin{equation}
	\label{eq:termI}
		\left|\mathbf{I}\right|\lesssim\left\{\sum_{T\in\CT_{h}}\eta_{T}^{2} \right\}^{1/2}\|((\btau,q),\bv)\|_{\mathbb{H}\times\mathbf{Q}},
	\end{equation}
and 
	\begin{equation}
	\label{eq:termII}
		\left|\mathbf{II}\right|\lesssim\left\{\sum_{T\in\CT_{h}}\eta_{T}^{2} \right\}^{1/2}\|((\btau,q),\bv)\|_{\mathbb{H}\times\mathbf{Q}}.
	\end{equation}
Since the calculations follow directly from these two references, we skip the details.

As a consequence of Lemma \ref{lema:apriorie}, Lemma \ref{lema_cota_s}, \eqref{eq:hot}, \eqref{eq:sup}, estimates \eqref{eq:termI}--\eqref{eq:termII},  and the definition of the local estimator $\eta_T$, we have the following result.
\begin{prop}
Let $(\lambda,(\bsig,p), \bu)\in\R\times\mathbb{H}\times \mathbf{Q}$ be the solution of \eqref{eq1}--\eqref{eq2}  and let $(\lambda_{h},(\bsig_h,p_h),\bu_h)\in\R\times\mathbb{H}_{h}\times\mathbf{Q}_h$ solution of  \eqref{eq1h}--\eqref{eq2h}. Then,  there exists $h_0$ for all $h < h_0$, there holds.
\begin{align*}
\|\bsig-\bsig_{h}\|_{\bdiv,\O}+\|p-p_h\|_{0,\O}+\|\bu-\bu_{h}\|_{0,\O}& \lesssim \left\{\sum_{T\in\CT_{h}}\eta_{T}^{2} \right\}^{1/2}+\|\bu-\Theta_h\bu\|_{0,\O},\\
|\lambda_{h}-\lambda |&\lesssim \sum_{T\in\CT_{h}}\eta_{T}^{2}+\|\bu-\Theta_h\bu\|_{0,\O}^2,
\end{align*}
where the hidden constant is independent of $h$.
\end{prop}
A similar result for the estimator $\theta$ is directly established if the pseudostress-velocity problem is considered. Therefore, the reliability of our estimators is guaranteed.
\subsection{Efficiency}
The  following task is to obtain a lower bound for the local indicator \eqref{eq:local_full}, which is obtained with a localization technique based in bubble functions, together with inverse inequalities. 


We begin by introducing the bubble functions for two and three dimensional elements. Given $T\in\mathcal{T}_h$ and $e\in\mathcal{E}(T)$, we let $\psi_T$ and $\psi_e$ be the usual element-bubble and facet-bubble functions, respectively (see \cite{MR3059294} for more details), satisfying the following properties
	\begin{enumerate}
		\item $\psi_T\in\mathrm{P}_{\ell}(T)$, with $\ell=3$ for 2D or $\ell=4$ for 3D, $\text{supp}(\psi_T)\subset T$, $\psi_T=0$ on $\partial T$ and $0\leq\psi_T\leq 1$ in $T$;
		\item $\psi_e|_T\in\mathrm{P}_{\ell}(T)$,  with $\ell=2$ for 2D or $\ell=3$ for 3D, $\text{supp}(\psi_e)\subset \omega_e:=\cup\{T'\in\mathcal{T}_h\,:\, e\in\mathcal{E}(T')\}$, $\psi_e=0$ on $\partial T\setminus e$ and $0\leq\psi_e\leq 1$ in $\omega_e$.
	\end{enumerate}

The  following result for bubble functions will be needed (see for instance \cite[Lemma 1.3]{MR1284252}).
\begin{lemma}[Bubble function properties]
\label{lmm:bubble_estimates}
Given $k\in\mathbb{N}\cup\{0\}$, and for each $T\in\mathcal{T}_h$ and $e\in\mathcal{E}(T)$, the hold
\begin{equation*}
\|\psi_T q\|_{0,T}^2\leq \|q\|_{0,T}^2\lesssim \|\psi_T^{1/2} q\|_{0,T}^2\quad\forall q\in\mathbb{P}_k(T),
\end{equation*}
\begin{equation*}
\|\psi_e L(p)\|_{0,e}^2\leq \| p\|_{0,e}^2\lesssim \|\psi_e^{1/2} p\|_{0,e}^2\quad\forall p\in\mathbb{P}_k(e),
\end{equation*}
and 
\begin{equation*}
h_e\|p \|_{0,e}^2\lesssim \|\psi_e^{1/2} L(p)\|_{0,T}^2\lesssim
 h_e\|p\|_{0,e}^2\quad\forall p\in\mathbb{P}_k(e),
 \end{equation*}
 where $L$ is the extension operator defined by  $L: C(e)\rightarrow C(T)$ where $C(e)$ and $C(T)$ are the spaces of continuous functions defined in $e$ and $T$, respectively. Also,   $L(p)\in\mathbb{P}_k(T)$ and $L(p)|_e=p$ for all $p\in\mathbb{P}_k(e)$ and where hidden constants depend on $k$ and the shape regularity of the mesh (minimum angle condition).
 \end{lemma}
We recall the following inverse inequality, proved in  \cite[Theorem 3.2.6]{MR0520174}.
\begin{lemma}[Inverse inequality]\label{inversein}
Let $l,m\in\mathbb{N}\cup\{0\}$ such that $l\leq m$. Then, for each $T\in\mathcal{T}_h$ there holds
\begin{equation*}
|q|_{m,T}\lesssim h_T^{l-m}|q|_{l,T}\quad\forall q\in\mathbb{P}_k(T),
\end{equation*}
where the hidden constant depends on $k,l,m$ and the shape regularity of the partition.
\end{lemma}

We also invoke the following two results. The first was proved in  \cite[Lemma 4.3]{MR2293249} and \cite[Lemma 4.9]{MR3453481} for the two and three dimensional cases, respectively. The second was proved in \cite[Lemma 4.10]{MR3453481}.
\begin{lemma}\label{lmm:invcurl}
Let $\btau_h\in\mathbb{L}^2(\O)$ be a piecewise polynomial of degree $k\geq 0$ on each $T\in\mathcal{T}_h$ such that approximates $\btau\in\mathbb{L}^2(\O)$, where  $\mathbf{curl}(\btau)=\boldsymbol{0}$ on each $T\in\mathcal{T}_h$. Then, there  holds
\begin{equation*}
\|\mathbf{curl}(\btau_h)\|_{0,T}\lesssim  h_T^{-1}\|\btau-\btau_h\|_{0,T}\quad \forall T\in\mathcal{T}_h,
\end{equation*}
where the hidden constant is  independent of $h$.
\end{lemma}
\begin{lemma}
\label{lmm:jump_control}
Let $\btau_h\in\mathbb{L}^2(\O)$ be a piecewise polynomial of degree $k\geq 0$ on each $T\in \CT_h$ and let $\btau\in\mathbb{L}^2(\O)$ be such that $\curl(\btau)=\boldsymbol{0})$ in $\O$. Then, there holds
\begin{equation*}
\left\|\jump{\btau_h\times\bn_e}\right\|_{0,e}\lesssim h_e^{-1/2}\|\btau-\btau_h\|_{0,\omega_e}\quad\forall e\in\mathcal{E}(\O),
\end{equation*}
where the hidden constant is independent of $h$.
\end{lemma}
Now our task is to bound each of the contributions of $\eta_T$ in \eqref{eq:local_full}. We begin with the term
$$h_T^2\left\|\curl\left\{\frac{1}{2\mu}\bsig_h^d\right\}\right\|_{0,T}^2.$$
Let us define $\btau=\bsig^{\dd}/(2\mu)$. Clearly $\curl(\btau)=\boldsymbol{0}$ since $\nabla\bu=\bsig^{\dd}/(2\mu)$. Define $\btau_h:=\bsig_h^{\dd}/(2\mu)$. Hence it is easy to obtain
\begin{equation*}
\|\btau-\btau_h\|_{0,T}\leq\frac{\sqrt{n}}{2\mu}\|\bsig-\bsig_h\|_{0,T}.
\end{equation*}
Then, applying Lemma \ref{lmm:invcurl} for $\btau$ defined as above, we obtain that 
\begin{equation}
\label{eq:bound1}
h_T^2\left\|\curl\left\{\frac{1}{2\mu}\bsig_h^d\right\}\right\|_{0,T}^2\lesssim \|\bsig-\bsig_h\|_{0,T}^2.
\end{equation}
Now, for  the term 
$$h_T^2\left\|\nabla\bu_h-\frac{1}{2\mu}\bsig_h^{\dd} \right\|^2_{0,T},$$
given  an element $T\in\CT_h$, let us define $\Upsilon_T:=\nabla\bu_h-\boldsymbol{\chi}_h$ where
$
\boldsymbol{\chi}_h:=\bsig_h^{\dd}/(2\mu).
$
Also we set 
$
\boldsymbol{\chi}:=\bsig^{\dd}/(2\mu).
$
 Invoking the estimate   $\Vert \tr(\bsig)\Vert_{0,T}\leq \sqrt{n}\Vert\bsig\Vert_{0,T}$ we obtain

\begin{equation*}
\Vert \boldsymbol{\chi}-\boldsymbol{\chi}_h\Vert_{0,T}\leq\frac{1}{2\mu}\left(\frac{n+\sqrt{n}}{n} \right)\|\bsig-\bsig_h\|_{0,T}.
\end{equation*}
%
%
Since $\bsig-2\mu\nabla\bu+p\mathbb{I}=\boldsymbol{0}$ (cf. \eqref{def:stokes_eigen}), it is clear that  $\nabla\bu=\boldsymbol{\chi}$. Then,  invoking the bubble function $\psi_T$, integrating by parts,  Cauchy-Schwarz inequality, Lemmas \ref{lmm:bubble_estimates} and \ref{inversein}, and the properties of $\psi_T$ given by Lemma \ref{lmm:bubble_estimates}, we obtain
\begin{align*}
\|\Upsilon_T\|_{0,T}^2&\lesssim \|\psi_T^{1/2}\Upsilon_T\|_{0,T}^2=\int_T\psi_T\Upsilon_T:(\nabla(\bu_h-\bu)+(\boldsymbol{\chi}-\boldsymbol{\chi}_h))\\
&= \int_T\bdiv(\psi_T\Upsilon_T)\cdot(\bu-\bu_h)+\int_T\psi_T\Upsilon_T:(\boldsymbol{\chi}-\boldsymbol{\chi}_h) \\
&\lesssim \left\{ h_T^{-1}\|\bu-\bu_h\|_{0,T}+\frac{n+\sqrt{n}}{2\mu n}\|\boldsymbol{\sigma}-\boldsymbol{\sigma}_h\|_{0,T} \right\}\|\Upsilon_T\|_{0,T}.
\end{align*}
 Hence
    \begin{equation}\label{eq:bound2}
    h_T^2\left\|\nabla\bu_h\hspace{-0.04cm}-\frac{1}{2\mu}\bsig_h^{\dd}  \right\|^2_{0,T}\lesssim\hspace{-0.04cm} \|\bu\hspace{-0.04cm}-\hspace{-0.04cm}\bu_h\|_{0,T}^2\hspace{-0.04cm}+\hspace{-0.04cm}h_T^2\|\bsig\hspace{-0.04cm}-\hspace{-0.04cm}\bsig_h\|_{0,T}^2,
    \end{equation}
    where the hidden constant is independent of $h$. Now we study the jump term 
\begin{equation*}
h_e\left\|\jump{\frac{1}{2\mu}\bsig_h^d\times\bn_e}\right\|_{0,e}^2.
\end{equation*}
To do this task,  set $\btau_h=\bsig_h^{\dd}/(2\mu)$ and  $\btau=\bsig^{\dd}/(2\mu)$ in Lemma \ref{lmm:jump_control} and from the definition of the deviator 
tensor, immediately we conclude
\begin{equation}
\label{eq:bound3}
h_e\left\|\jump{\frac{1}{2\mu}\bsig_h^d\times\bn_e}\right\|_{0,e}^2\lesssim \|\bsig-\bsig_h\|_{0,\omega_e}^2.
\end{equation}

All the previous terms are related to $\theta_T$, which is a part of $\eta_T$. Now we bound the rest of the terms. We begin with 
\begin{equation*}
h_T^2\left\|\curl\left( p_h+\frac{1}{n}\tr(\bsig_h)\mathbb{I}\right)\right\|_{0,T}^2.
\end{equation*}
In fact, setting $\btau_h=p_h+(1/n)\tr(\bsig_h)\mathbb{I}$ and $\btau=p+(1/n)\tr(\bsig)\mathbb{I}$ on Lemma \ref{lmm:invcurl} and noticing that
\begin{equation*}
\|\btau-\btau_h\|_{0,T}\lesssim \|p-p_h\|_{0,T}+\|\bsig-\bsig_h\|_{0,T},
\end{equation*}
we immediately obtain 
\begin{equation}
\label{eq:bound4}
h_T^2\left\|\curl\left( p_h+\frac{1}{n}\tr(\bsig_h)\mathbb{I}\right)\right\|_{0,T}^2\lesssim \|p-p_h\|_{0,T}^2+\|\bsig-\bsig_h\|_{0,T}^2.
\end{equation}
Then, gathering \eqref{eq:bound1},  \eqref{eq:bound2}, \eqref{eq:bound3} and \eqref{eq:bound4}
we prove the following result.
\begin{theorem}[Efficiency]
The following estimate holds

\begin{equation*}
\theta^2:=\sum_{T\in \CT_h}\theta_T^2\lesssim (\|\bu-\bu_h\|_{0,\O}^2+\|\bsig-\bsig_h\|_{0,\O}^2+\text{h.o.t.}),
\end{equation*}
and hence
$$
\eta^2:=\sum_{T\in\CT_{h}}\eta_{T}^{2} \lesssim \left( \|\bu-\bu_h\|_{0,\Omega}^2 + \|\bsig-\bsig_h\|_{0,\Omega}^2+\|p-p_h\|_{0,\O}^2+\text{h.o.t}\right),
$$
where the hidden constants are independent of $h$ and the discrete solutions.
\end{theorem}

\section{Numerical experiments}
\label{sec:numerics}
The aim of this section is to confirm, computationally, that the proposed method works correctly and delivers an accurate approximation of the spectrum 
of $\bT$. Moreover, we will confirm the theoretical results with the computation of the convergence order by means of a least-square fitting or a sufficiently accurate solution. The reported results have been obtained with a FEniCS code \cite{MR3618064}, together with the mesh generator Gmsh \cite{geuzaine2009gmsh}.

Throughout this section, we denote by $N$ the number of degrees of freedom and $\lambda$ the eigenvalues. We denote by $\err_f(\lambda_i)$ and $\err_r(\lambda_i)$  the errors on the $i$-th eigenvalue using the pseudostresss-pressure-velocity and pseudostress-velocity scheme, respectively, whereas the effectivity indexes with respect to $\eta$ or $\theta$ and the eigenvalue $\lambda_i$ are defined by
$$
\eff_f(\lambda_i):=\frac{\err_f(\lambda_i)}{\eta^2},\quad \eff_r(\lambda_i):=\frac{\err_r(\lambda_i)}{\theta^2}.
$$

In order to apply the adaptive finite element method,  we use blue-green marking strategy to refine each $T'\in \CT_{h}$ whose indicator $\beta_{T'}$ satisfies
$$
\beta_{T'}\geq 0.5\max\{\beta_{T}\,:\,T\in\CT_{h} \},
$$
where $\beta_{T}$ corresponds to either local estimator $\theta_{T}$ or $\eta_{T}$.

We divide each numerical tests in two parts: one related to the performance of the estimator \eqref{eq:global_est} and the second for \eqref{eq:global_full}.

\subsubsection{Test 1. Comparison between finite element families.} 
The  aim of this tests is to show the performance of the adaptive scheme when the $\mathbb{RT}_0$ and the lowest order $\mathbb{BDM}$ families are considered to approximate the pseudostress tensor. Note that in \cite{lepe2020priori} it is stated that $\mathbb{BDM}$ results in more stable uniform approximations, however, the computational cost makes the work with $\mathbb{RT}_0$ the most suitable. The domain for this test is an L-shaped $\Omega=(-1,1)\times(-1,1)\backslash((-1,0)\times(-1,0))$. In figure \ref{fig:l-shape-initial-mesh} we show the initial mesh. Note that for this problem we have a re-entrant corner in $(0,0)$, so the expected order of convergence for the eigenvalues is at least $\mathcal{O}(h^{r})$, with $r\geq 1.2$. For this test we have performed $20$ adaptive iterations in order to observe the convergence rate, as well as the refinement around the singularity.

Figure \ref{fig:lshape-error} shows the error curves in the adaptive refinements. An optimal convergence rate $\mathcal{O}(N^{-1})\simeq\mathcal{O}(h^{2})$ is clearly observed. In addition, we note that the adaptive iterations using $\mathbb{BDM}$ mark much fewer elements than with $\mathbb{RT}$. This is explained by the additional degrees of freedom that $\mathbb{BDM}$ has, achieving better approximations in each iteration, thus reducing the contribution of each local estimator. However, the error curves show more pronounced oscillations with respect to the $\mathbb{RT}_0$ approach, probably caused by an under-prediction in the local residual contributions. It is for this reason that in the rest of the experiments we will use $\mathbb{RT}_0$ to approximate $\bsig_h$.

We finish the test with Figure \ref{fig:lshape-mesh-estimador-theta}, where we observe the meshes at iteration 15 when both families are used. The difference in the number of refined elements is evident. 
 \begin{figure}
 \centering
 \includegraphics[scale=0.06,trim= 45cm 2.5cm 35cm 2.5cm,clip]{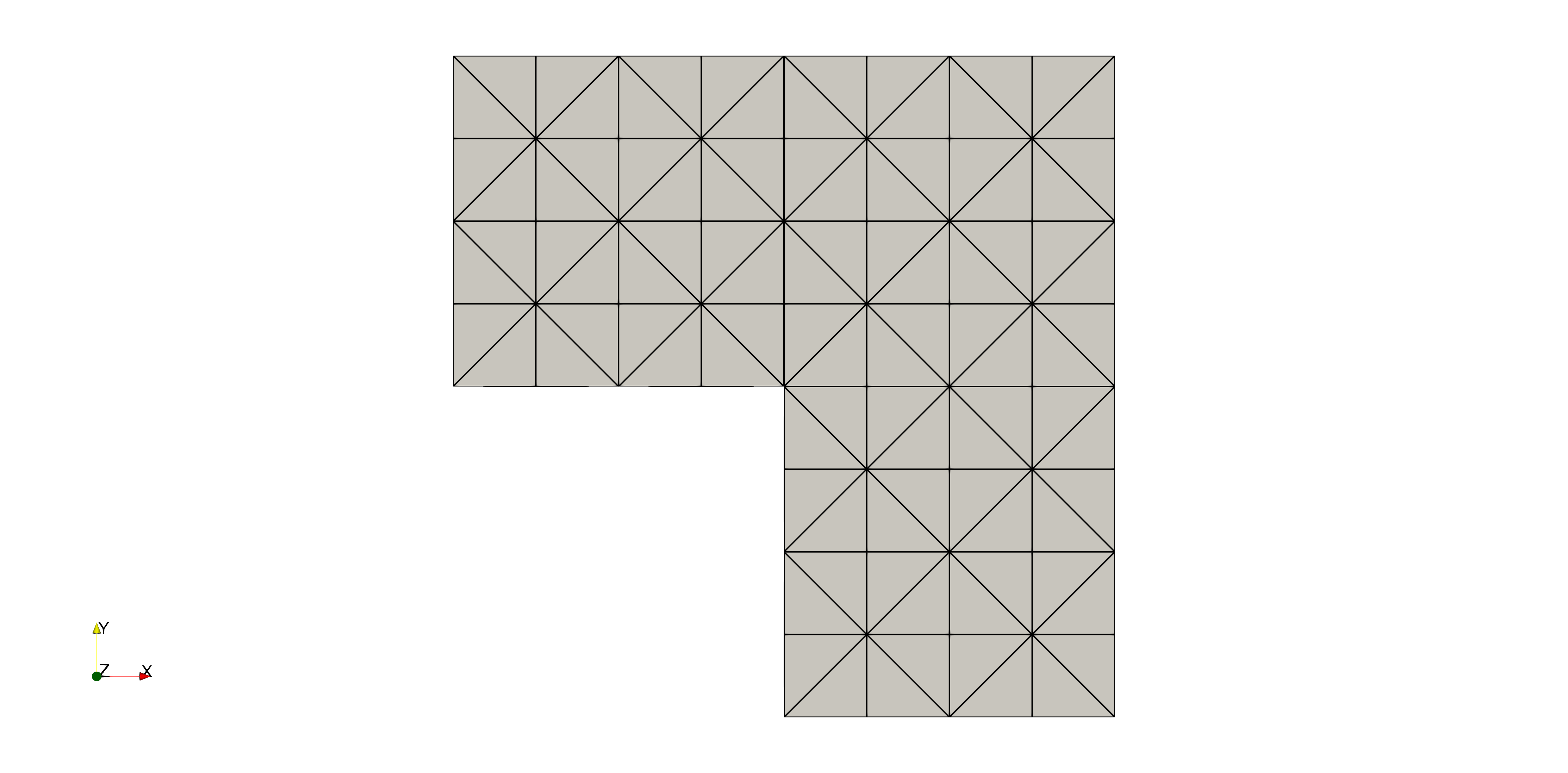}
 \caption{Test 1. Initial mesh configuration.}
 \label{fig:l-shape-initial-mesh}
 \end{figure}
	\begin{figure}
	\centering
	\includegraphics[scale=0.25]{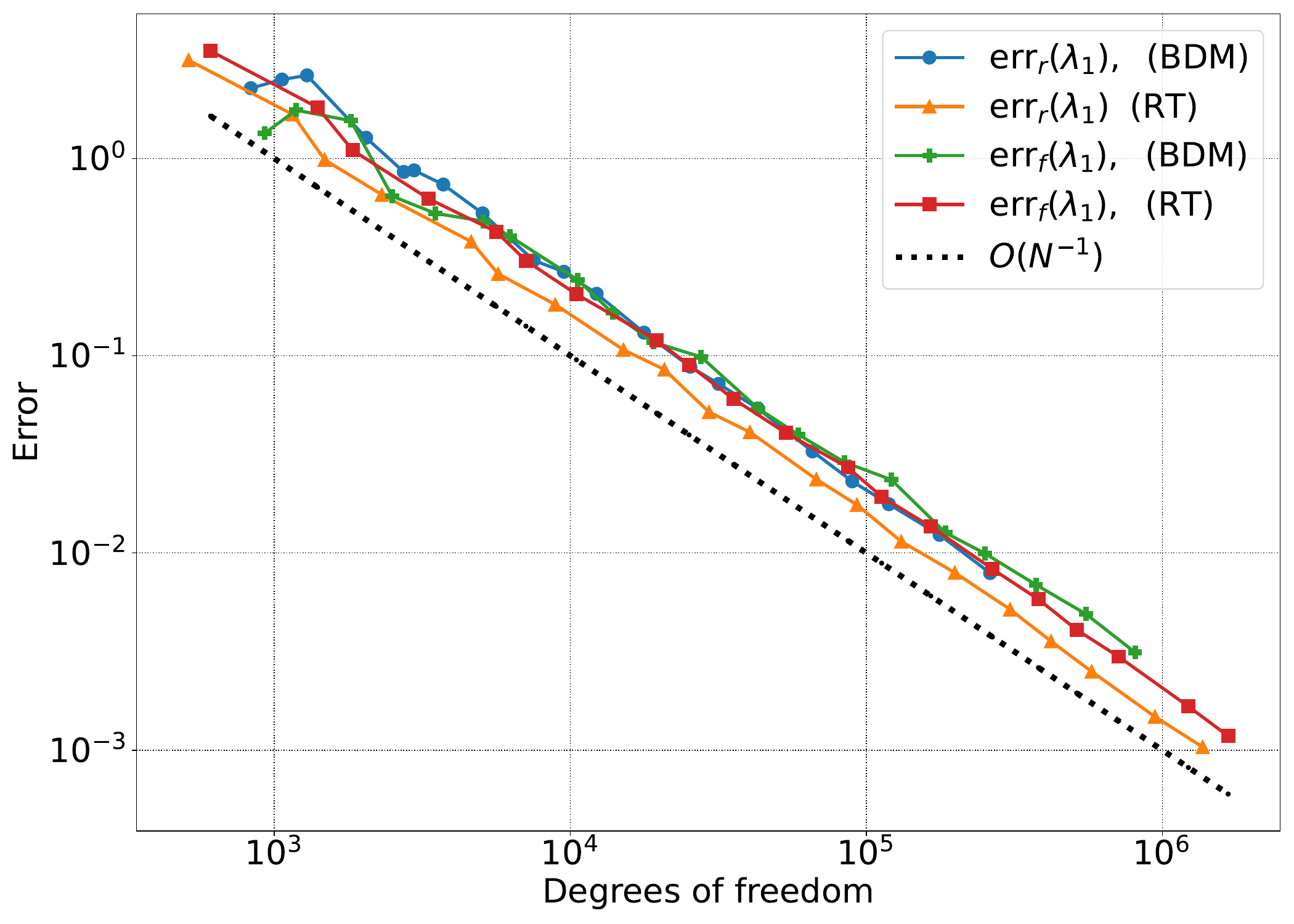}
		%
	\caption{Test 1. Error curves when using $\theta$ and $\eta$ as estimators in the two dimensional L-shaped domain, and using $\mathbb{RT}_0$ and $\mathbb{BDM}$ to approximate $\bsig_{h}$.}
	\label{fig:lshape-error}
\end{figure}
\begin{figure}
	\centering
	\begin{minipage}{0.48\linewidth}
			\includegraphics[scale=0.06,trim= 45cm 2.5cm 35cm 2.5cm, clip]{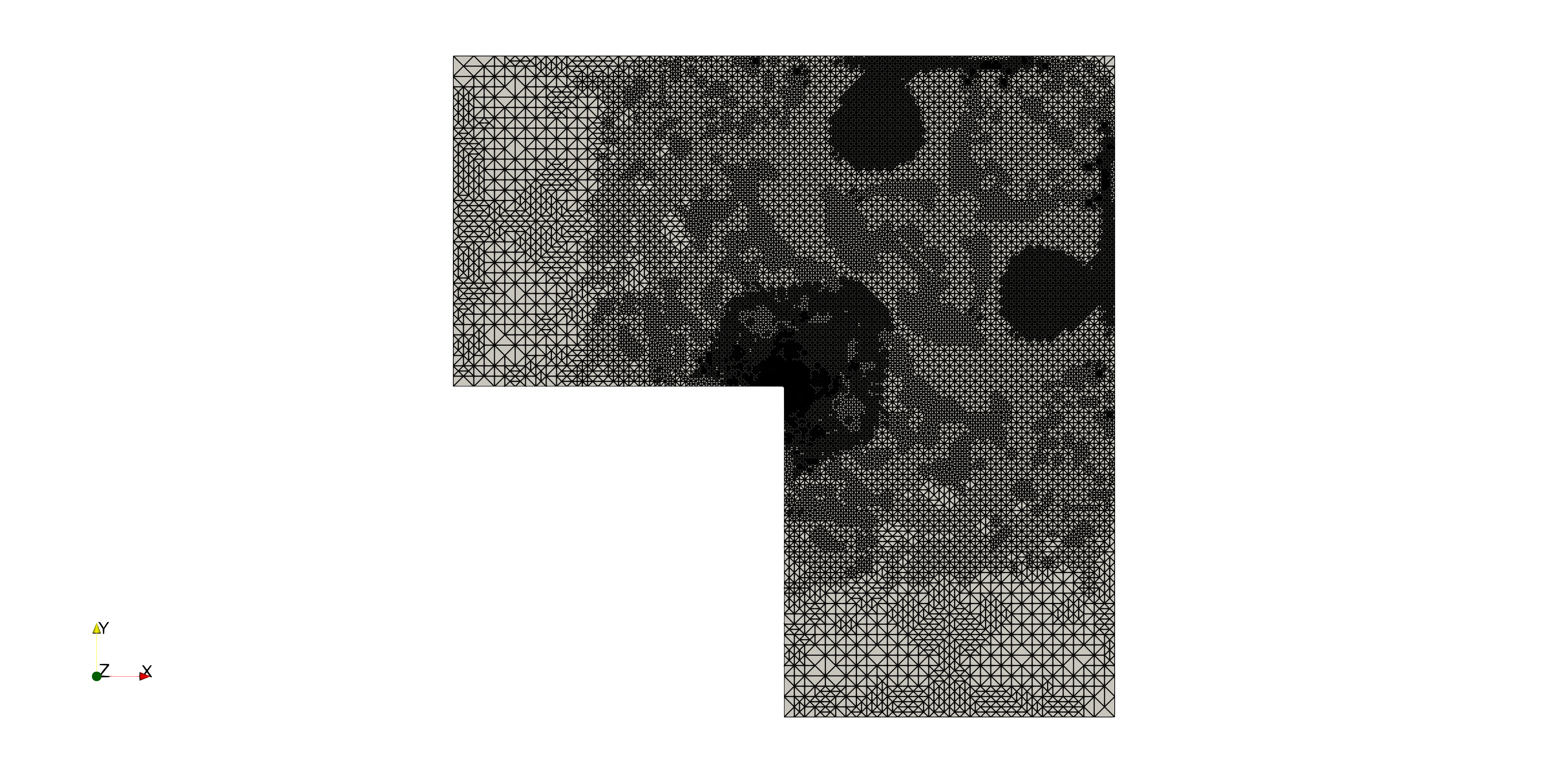}
		\end{minipage}
	\begin{minipage}{0.48\linewidth}
			\includegraphics[scale=0.06,trim= 45cm 2.5cm 35cm 2.5cm,, clip]{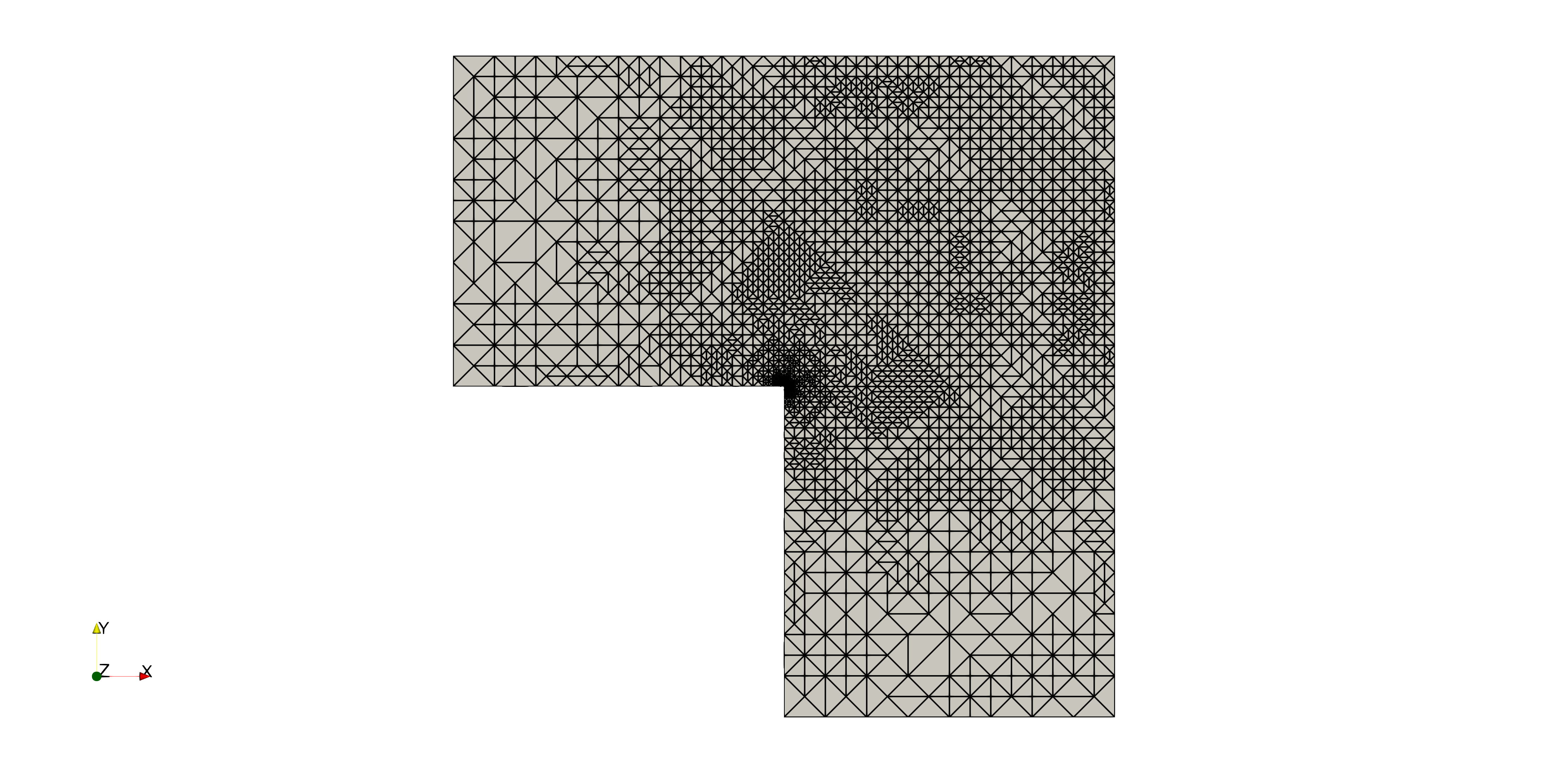}
		\end{minipage}\\
	\begin{minipage}{0.48\linewidth}
			\includegraphics[scale=0.06,trim= 45cm 2.5cm 35cm 2.5cm,, clip]{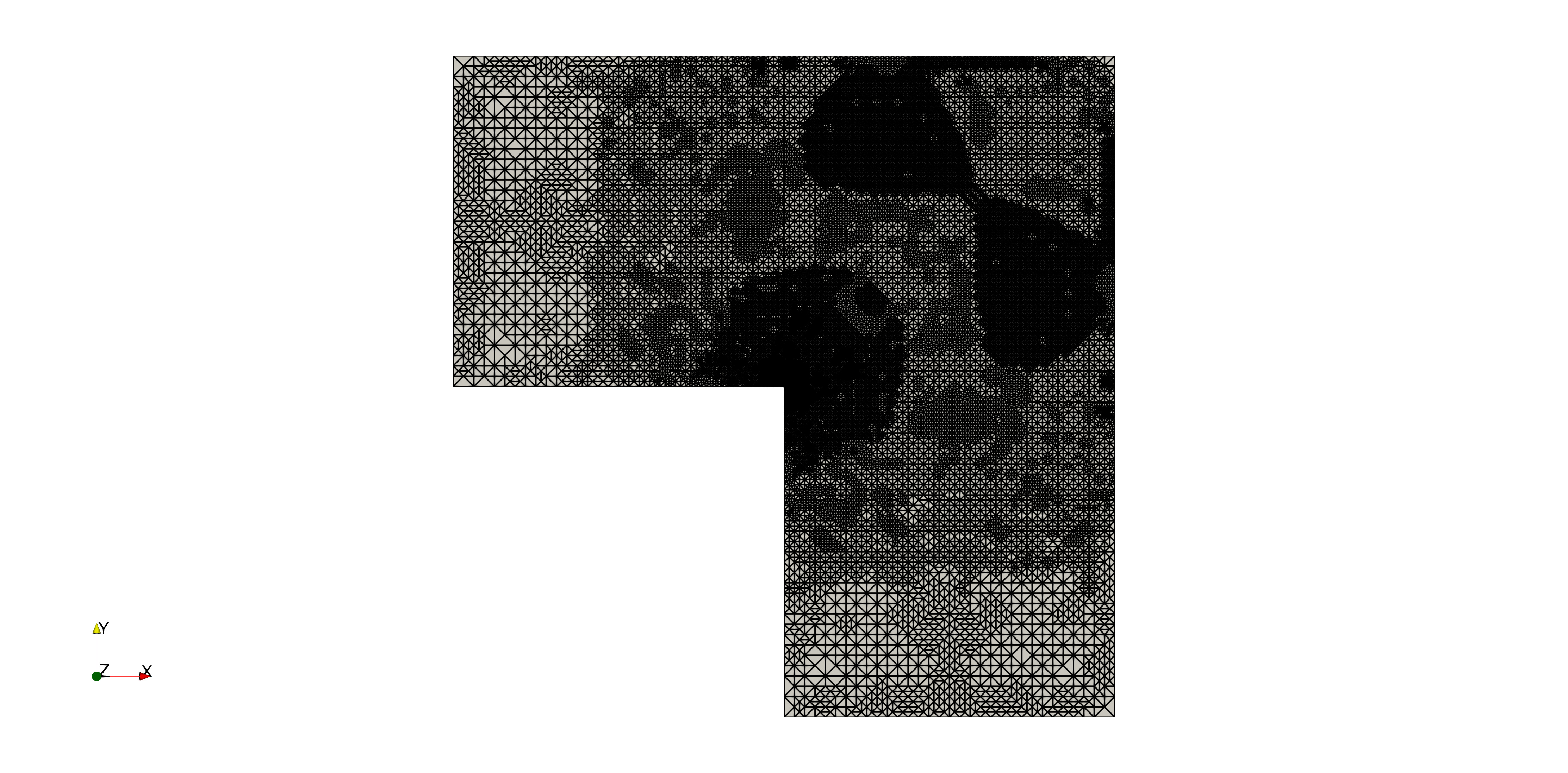}
		\end{minipage}
	\begin{minipage}{0.48\linewidth}
			\includegraphics[scale=0.06,trim= 45cm 2.5cm 35cm 2.5cm,, clip]{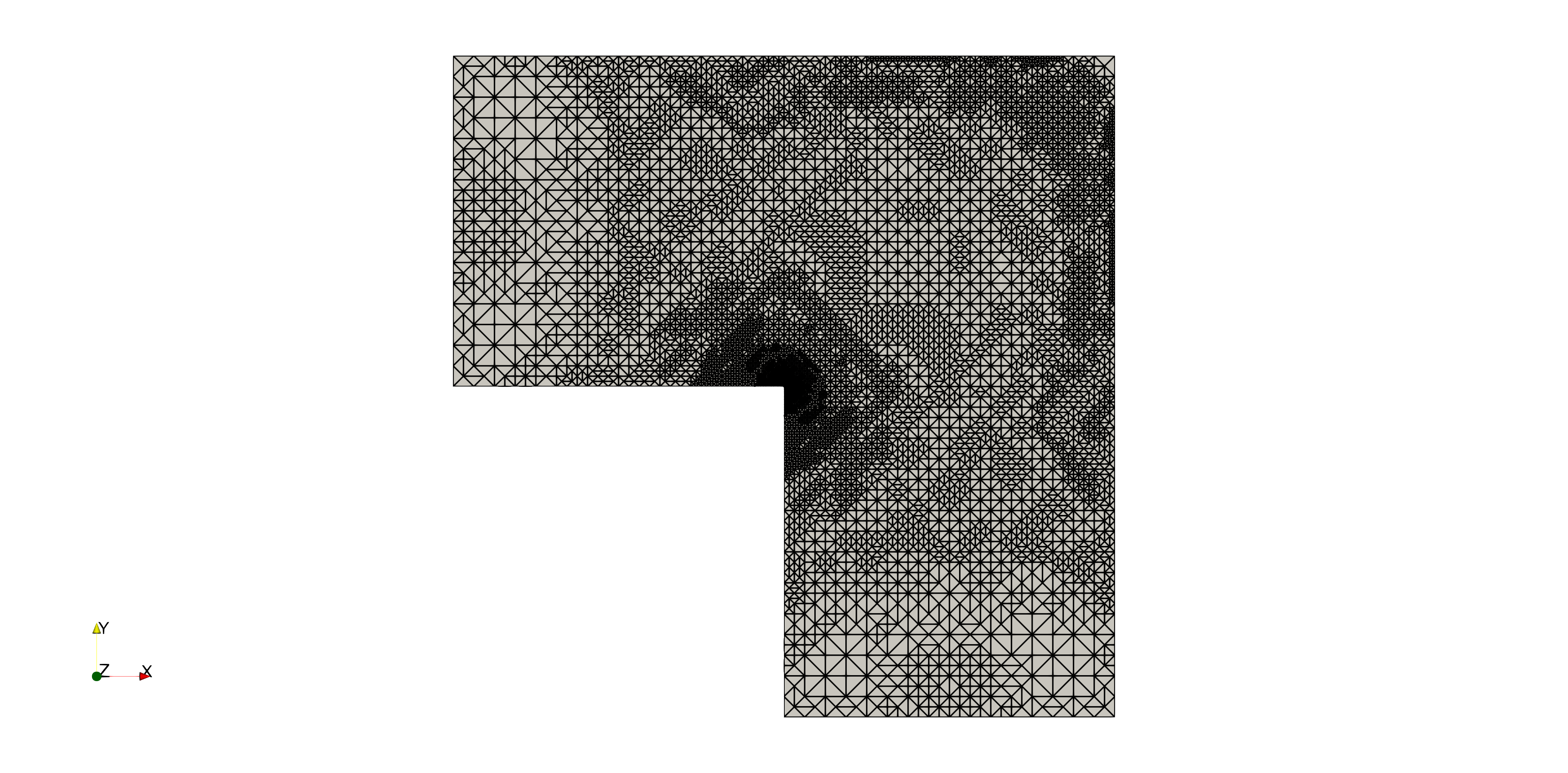}
		\end{minipage}
	\caption{Test 1. Adaptive meshes in the fifteenth iteration using $\mathbb{RT}_0$ (left column) and $\mathbb{BDM}$ (right column) to approximate $\bsig_{h}$. Top row: meshes using the estimator $\theta$. Bottom row: meshes using estimator  $\eta$.}
	\label{fig:lshape-mesh-estimador-theta}
\end{figure}
\subsubsection{Test 2. The T-shaped domain.} This test aims to confirm that our estimators are able to detect and refine close to the singularity of the domain in order to recover the optimal order of convergence. The domain is defined as $$\Omega:=(-1,1)^2\backslash\big((-1,-1/3)\times(-1,1/2)\cup(1/3,1)\times(-1,1/2)\big).$$

In Figure \ref{fig:t-shape-initial-mesh} we show the initial mesh for this domain. Note that for this geometrical configuration we have two re-entrant corners in $(-1/3,1/2)$ and $(1/3,1/2)$, so the expected order of convergence for the lowest order eigenvalue is roughly $\mathcal{O}(N^{-0.66})\simeq\mathcal{O}(h^{1.32})$ (see, for instance, \cite{lepe2020priori}). Let us remark that on this test we have performed $15$ adaptive iterations in order to observe the convergence rate, as well as the refinement around the singularities.
Table \ref{tabla:tshape-2D-uniform-vs-adaptive} shows the comparative of the uniform and adaptive refinement when using both numerical schemes. It notes that the computed order of convergence using uniform refinements is approximately $\mathcal{O}(h^{1.32})$. Also, we observe that with a fraction of $1/7$ of the degrees of freedom in the uniform refinements, the adaptive numerical schemes approximate the extrapolated eigenvalue with high accuracy. 

In Figure \ref{fig:tshape-mesh-estimador-theta} (top row) we observe several intermediate meshes when we solve the eigenvalue problem using the estimator $\theta$ and $\eta$. Note that the estimators refine near the high pressure gradients. 
Error curves of the two numerical schemes are observed in Figure \ref{fig:tshape-error}, where we observe that the optimal order of convergence is recovered.

In Table \ref{tabla:tshape-2D-scheme1-vs-scheme2} a comparison between the errors and effectivity indexes is reported. Here, we note that both schemes gives similar error behavior, whereas the estimators values suggest that the contribution of the residuals in the pseudostress-pressure-velocity will yield to a different marking of the elements near the singularity. This is contrasted with the adaptive meshes in Figure \ref{fig:tshape-mesh-estimador-theta}, where we observe that the estimator $\eta$ marks more elements around the singularities than $\theta$.

We finish this test by showing the lowest computed eigenfunctions when using the pseudostress-pressure-velocity scheme. The velocity field and pressure contour lines are depicted in Figure \ref{fig:tshape-u-p}.

\begin{table}
	{\footnotesize
		\caption{Test 2: Comparison between the lowest computed eigenvalue in the pseudostress-velocity scheme using uniform and adaptive refinements.}
		\label{tabla:tshape-2D-uniform-vs-adaptive}
		\begin{center}
			\begin{tabular}{l c c c  |l c c c }
				\toprule
				\multicolumn{4}{c}{Pseudostress-velocity} & \multicolumn{4}{|c}{Pseudostress-pressure-velocity}\\
				\midrule
				\multicolumn{2}{c}{Uniform} &\multicolumn{2}{c}{Adaptive}& \multicolumn{2}{|c}{Uniform} & \multicolumn{2}{c}{Adaptive}\\
				\midrule
				$N$&$\lambda_{h1}$ & $N$&$\lambda_{h1}$&$N$&$\lambda_{h1}$ & $N$&$\lambda_{h1}$ \\ 
				\midrule
				597		& 59.07677  &597&		 59.07677  &709		& 57.92345 &709		& 57.92345  \\
				2313	& 73.04426 &1113		& 70.76256 &2761		& 72.62093 &1291	& 68.71842   \\
				9105	& 77.95812 &2125		& 75.64725 &10897		& 77.83306 &2517	& 74.86457 \\
				36129	& 79.73824 &3779		& 78.23674 &43297		& 79.70449 &4887	& 77.84735 \\
				143937	& 80.40890 &7477		& 79.53645 &172609	& 80.40016 &7145	& 78.97951 \\
				574593	& 80.67726 &10609		& 80.01029 &689281	& 80.67504 &9495	& 79.55730  \\
				2296065	& 80.80075 &16887		& 80.39234 &2754817	& 80.80019 &15637	& 80.19072 \\
				&  		&25239	   & 80.53787&& 			  		   &23815	& 80.36188 \\
				&  		&35005	   & 80.64562&& 			 		   &33865	& 80.54858\\
				&  		&50231		& 80.72616&& 			 		&47179		& 80.65184 \\
				&  		&78023	 	& 80.78096&& 			 		&75235	 	& 80.73734\\
				&  		&108877	& 80.80975&& 			 		&111191	& 80.78710  \\
				&  		&150463	& 80.83179&& 			  		&150976	& 80.81563 \\
				&  		&224137	& 80.84940&& 			  		&224587	& 80.83532 \\
				&  		&332507	& 80.86030&& 			  		&343874	& 80.85180 \\
				\midrule
				Order	&$\mathcal{O}(N^{-0.67})$		&Order	&$\mathcal{O}(N^{-1.10})$& Order	&$\mathcal{O}(N^{-0.68})$		&Order	&$\mathcal{O}(N^{-1.09})$ \\
				$\lambda_1$	& 80.87944 		&$\lambda_1$	& 80.87944 & $\lambda_1$	& 80.87944 		&$\lambda_1$	& 80.87944 \\
				\bottomrule             
			\end{tabular}
	\end{center}}
\end{table}

\begin{table}[H]
	{\footnotesize
		\caption{Test 2: Computed errors, estimators and effectivity indexes on the adaptively refinement meshes for each numerical scheme.}
		\label{tabla:tshape-2D-scheme1-vs-scheme2}
		\begin{center}
			\begin{tabular}{c c c c|c c c c }
				\toprule
				\multicolumn{3}{c}{Pseudostress-velocity} &&& \multicolumn{3}{c}{Pseudostress-pressure-velocity}\\
				\midrule
				$\err_r(\lambda_{h1})$&$\theta^2$&$\eff_r(\lambda_{h1})$ &&& $\err_f(\lambda_{h1})$&$\eta^2$&$\eff_f(\lambda_{h1})$ \\ 
				\midrule
				2.18027e+01	&1.71691e+02& 1.26988e-01 &&&2.29560e+01	&2.01020e+02& 1.14198e-01  \\
				1.01169e+01	&1.00435e+02& 1.00730e-01 &&&1.21610e+01	&1.17826e+02& 1.03212e-01  \\
				5.23219e+00	&6.76280e+01& 7.73672e-02 &&&6.01486e+00	&7.64312e+01& 7.86964e-02 \\
				2.64269e+00	&3.61802e+01& 7.30426e-02 &&&3.03209e+00	&4.05116e+01& 7.48449e-02  \\
				1.34298e+00	&1.87337e+01& 7.16879e-02 &&&1.89993e+00	&2.63974e+01& 7.19739e-02  \\
				8.69150e-01	&1.31928e+01& 6.58808e-02 &&&1.32213e+00	&1.99707e+01& 6.62036e-02  \\
				4.87092e-01	&8.55359e+00& 5.69459e-02 &&&6.88721e-01	&1.25131e+01& 5.50399e-02 \\
				3.41569e-01 &5.61139e+00& 6.08708e-02 &&&5.17553e-01	&8.09004e+00& 6.39740e-02 \\
				2.33812e-01 &4.08210e+00& 5.72774e-02 &&&3.30857e-01	&5.65293e+00& 5.85284e-02 \\
				1.53280e-01 &2.93354e+00& 5.22508e-02 &&&2.27595e-01	&4.16076e+00& 5.47005e-02 \\
				9.84765e-02 &1.87021e+00& 5.26554e-02 &&&1.42101e-01	&2.62431e+00& 5.41479e-02 \\
				6.96847e-02 &1.31854e+00& 5.28498e-02 &&&9.23345e-02	&1.73254e+00& 5.32944e-02  \\
				4.76441e-02 &9.63975e-01& 4.94246e-02 &&&6.38039e-02	&1.28003e+00& 4.98456e-02 \\
				3.00353e-02 &6.64955e-01& 4.51688e-02 &&&4.41141e-02	&8.84236e-01& 4.98895e-02 \\
				1.91353e-02 &4.40433e-01& 4.34466e-02 &&&2.76378e-02	&5.70886e-01& 4.84121e-02  \\
				\bottomrule             
			\end{tabular}
	\end{center}}
\end{table}

\begin{figure}
	\centering
	\includegraphics[scale=0.06,trim= 35cm 2.5cm 35cm 2.5cm,clip]{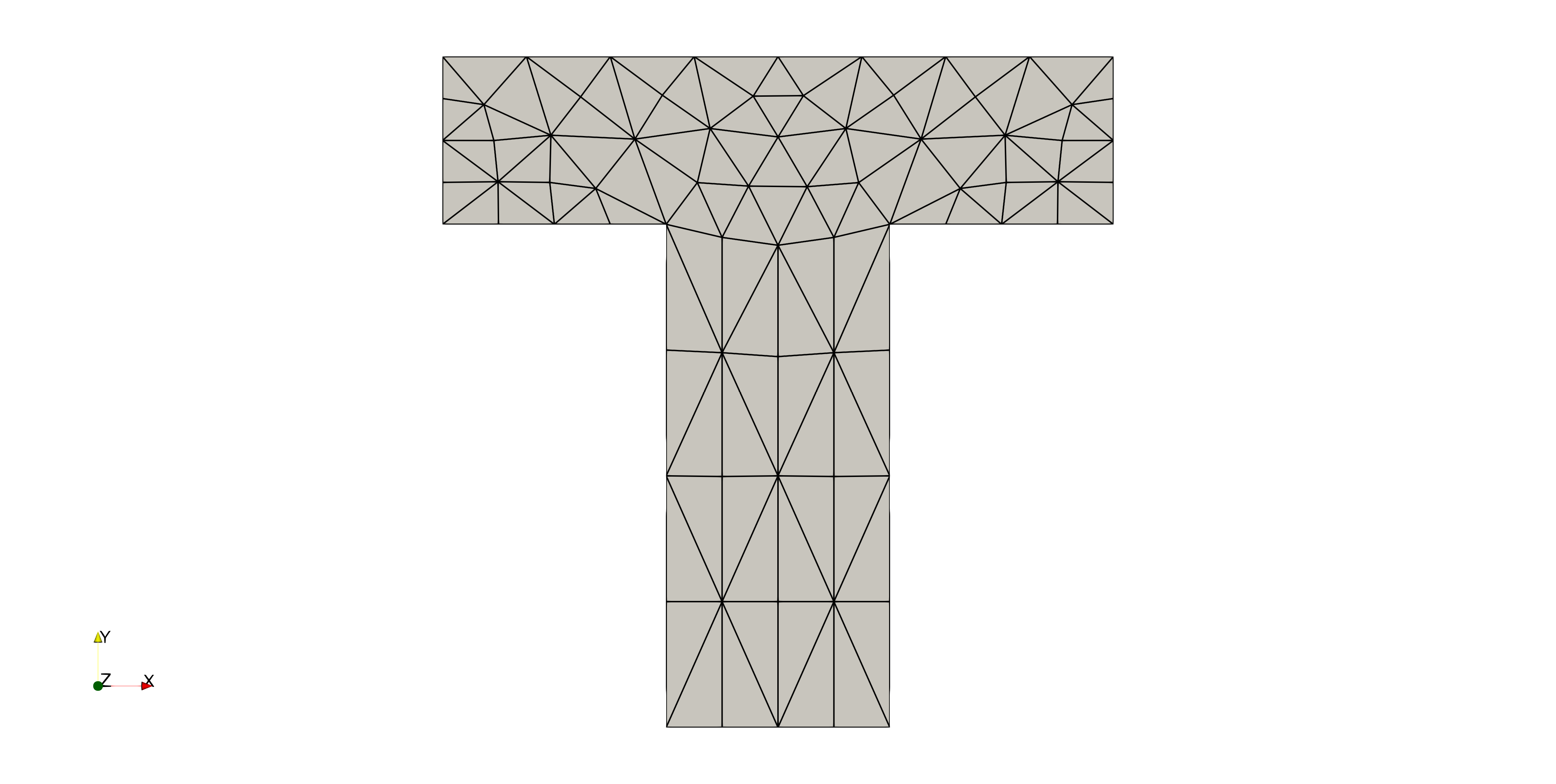}
	\caption{Test 2. Initial mesh configuration.}
	\label{fig:t-shape-initial-mesh}
\end{figure}

\begin{figure}
	\centering
	\begin{minipage}{0.48\linewidth}
		\includegraphics[scale=0.06,trim= 35cm 2.5cm 35cm 2.5cm,, clip]{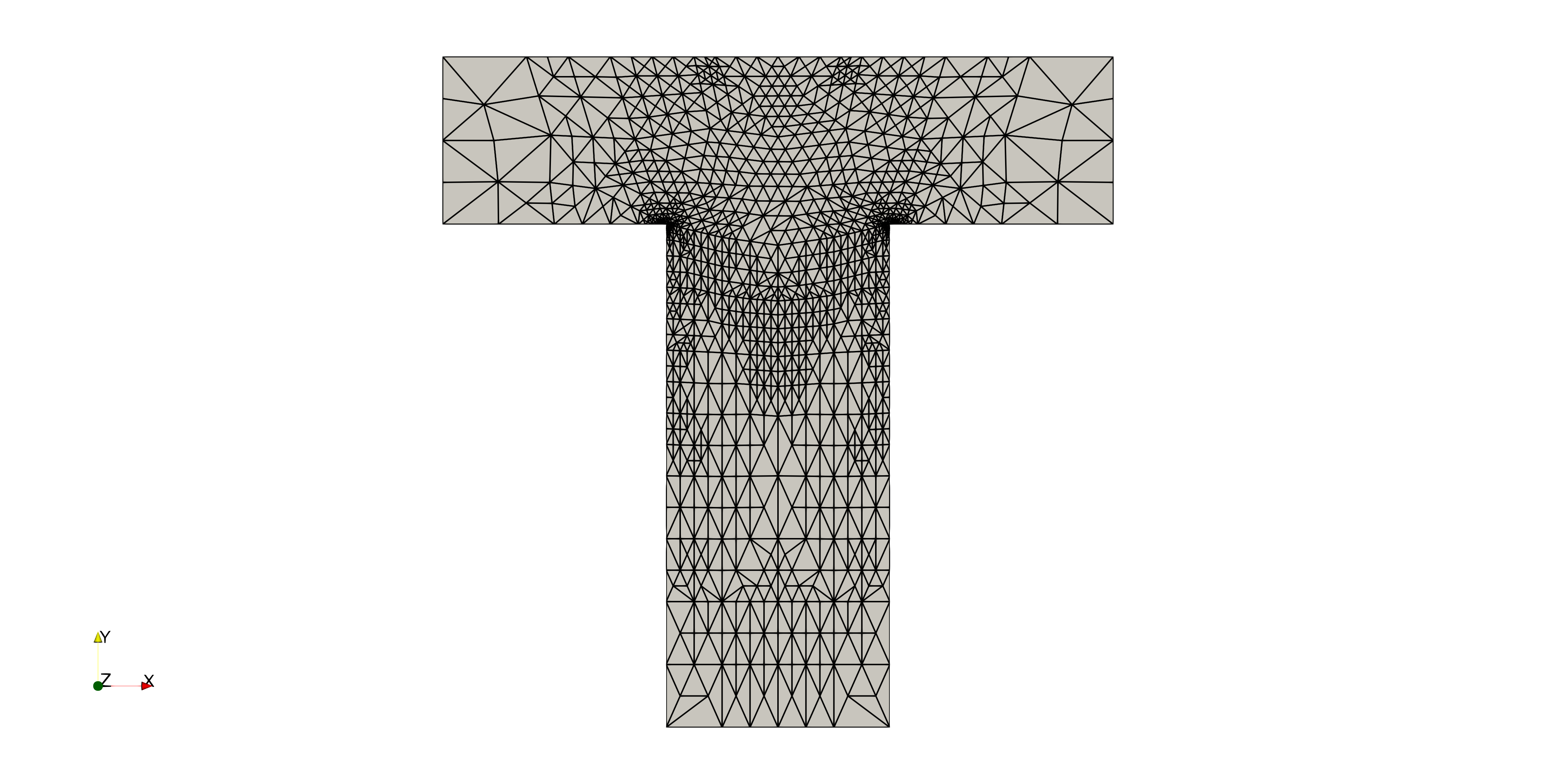}
	\end{minipage}
	\begin{minipage}{0.48\linewidth}
		\includegraphics[scale=0.06,trim= 35cm 2.5cm 35cm 2.5cm,, clip]{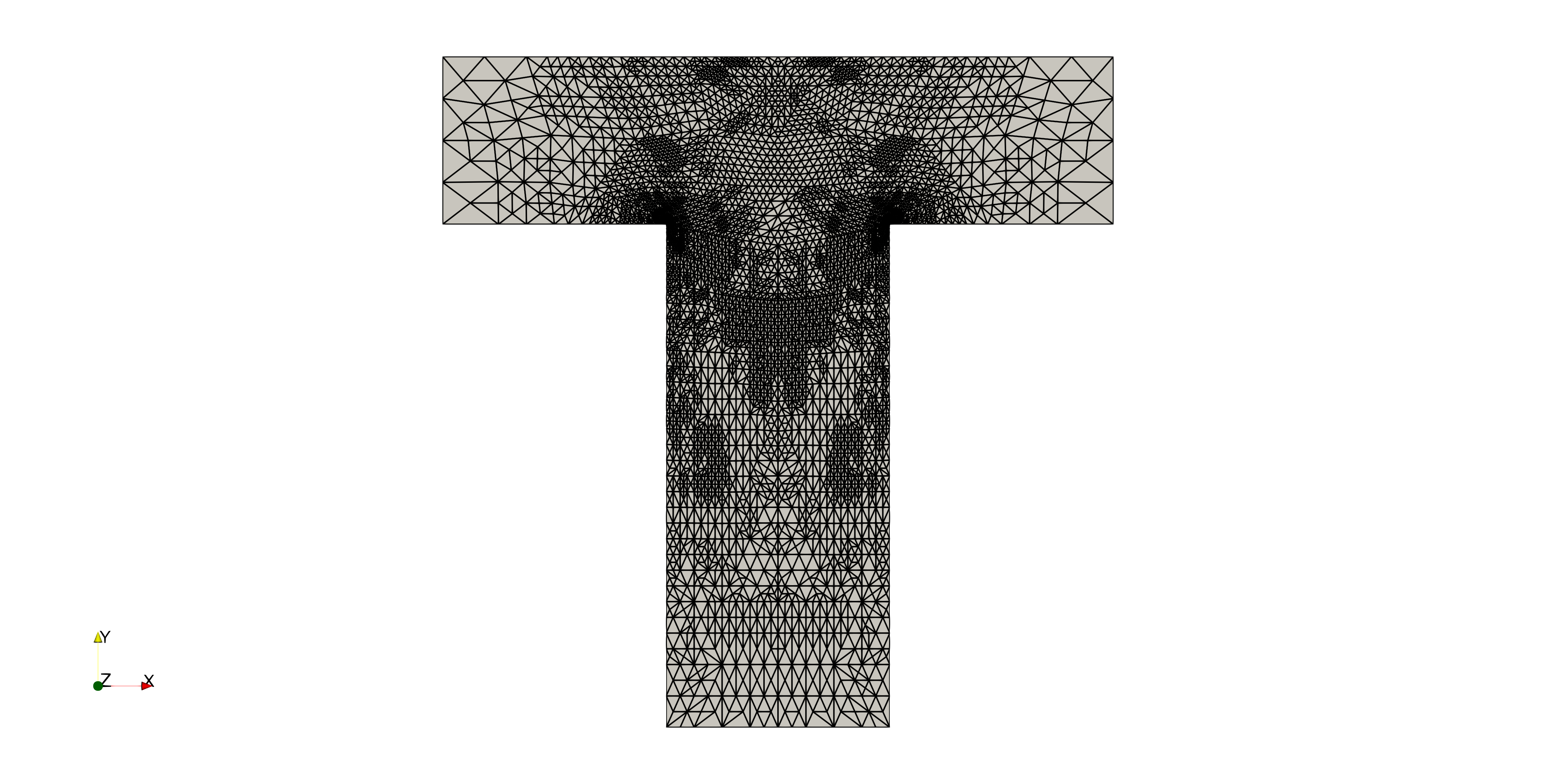}
	\end{minipage}\\
\begin{minipage}{0.48\linewidth}
	\includegraphics[scale=0.06,trim= 35cm 2.5cm 35cm 2.5cm,, clip]{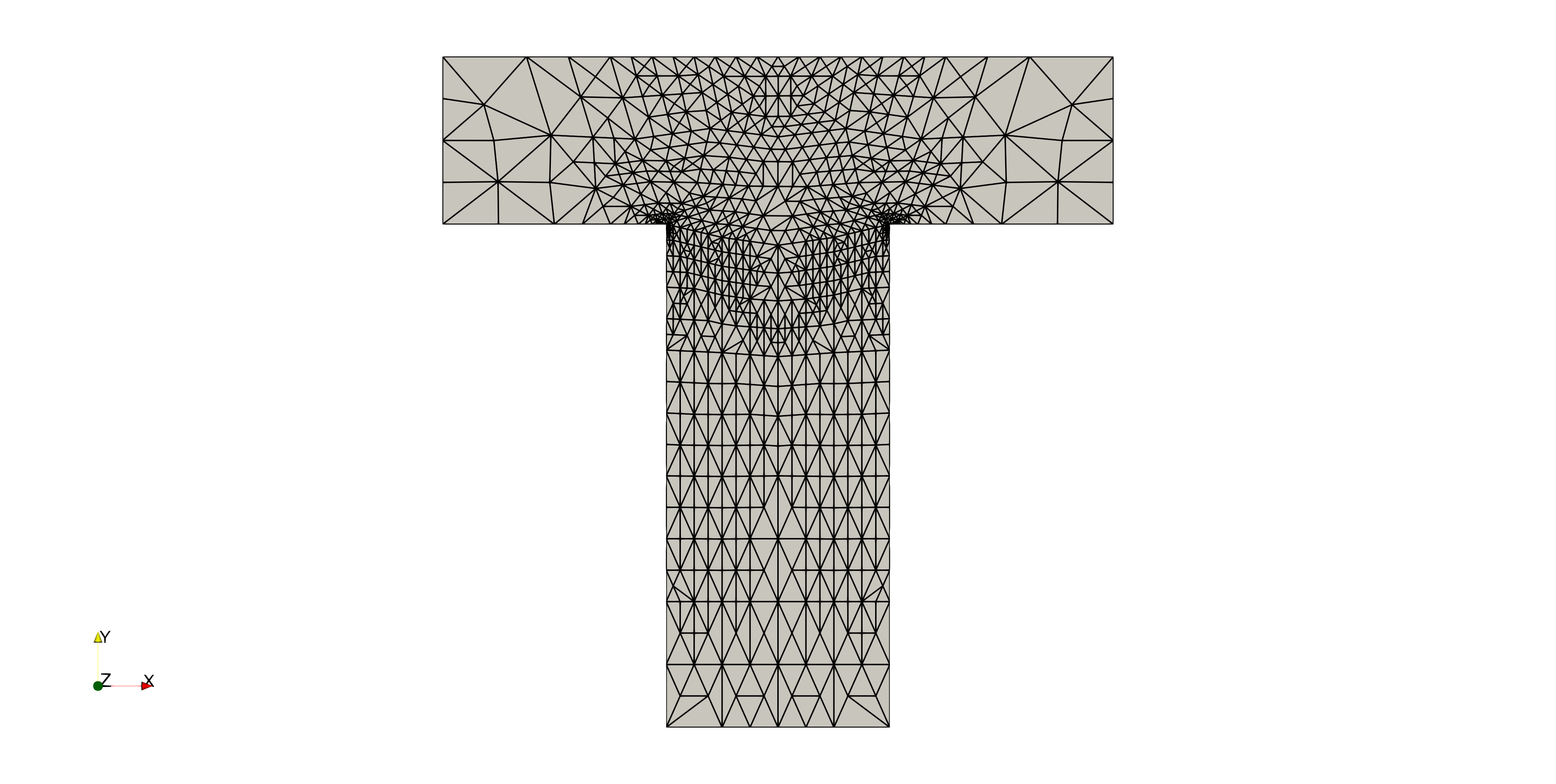}
\end{minipage}
\begin{minipage}{0.48\linewidth}
	\includegraphics[scale=0.06,trim= 35cm 2.5cm 35cm 2.5cm,, clip]{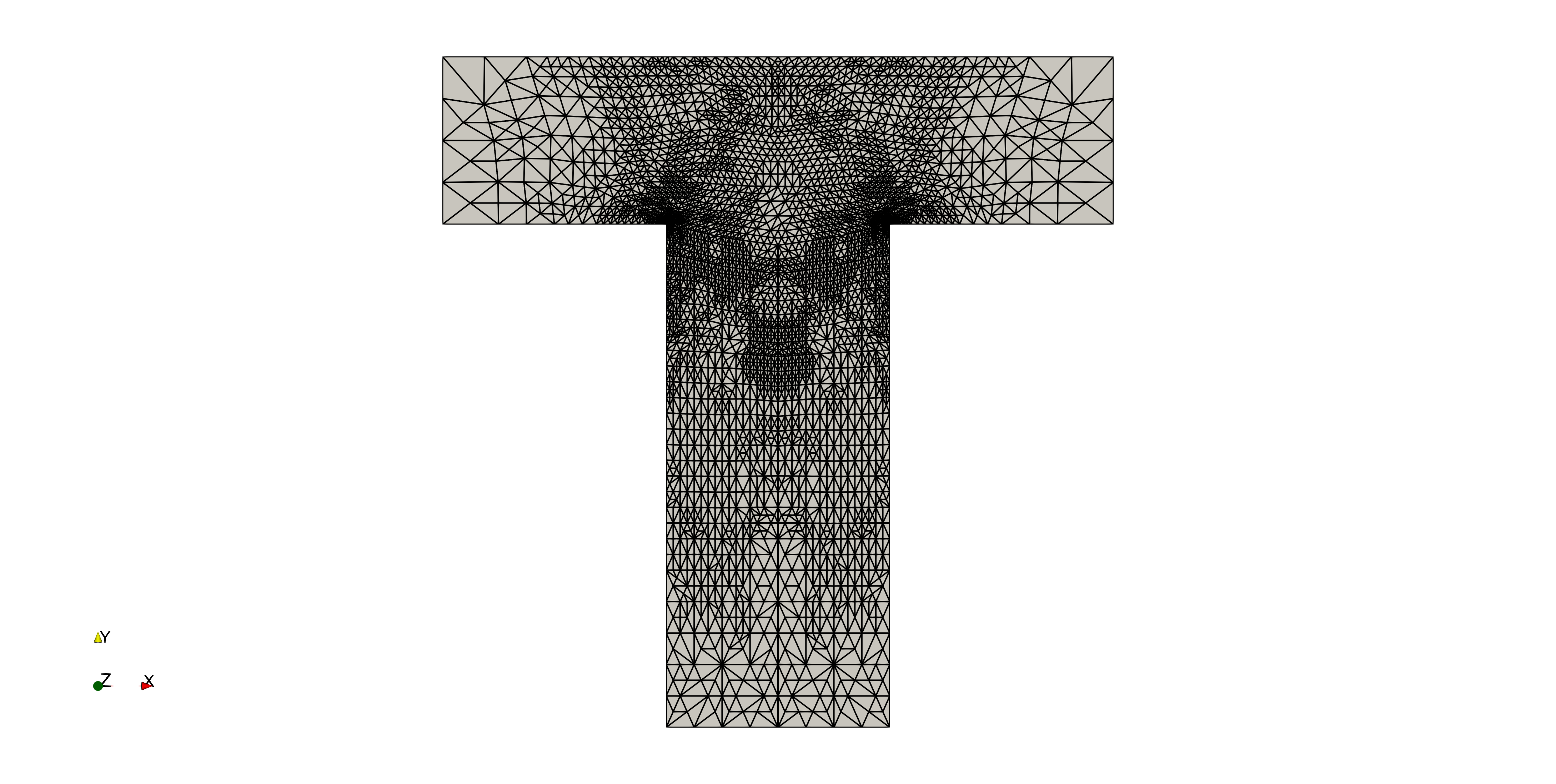}
\end{minipage}
	\caption{Test 2. Intermediate adaptive meshes. Top row: meshes with $10609$ and  $50231$ degrees of freedom using the estimator $\theta$. Bottom row: meshes with $9495$ and $47179$ degrees of freedom using the estimator $\eta$. }
	\label{fig:tshape-mesh-estimador-theta}
\end{figure}
	\begin{figure}
	\centering
	\includegraphics[scale=0.25]{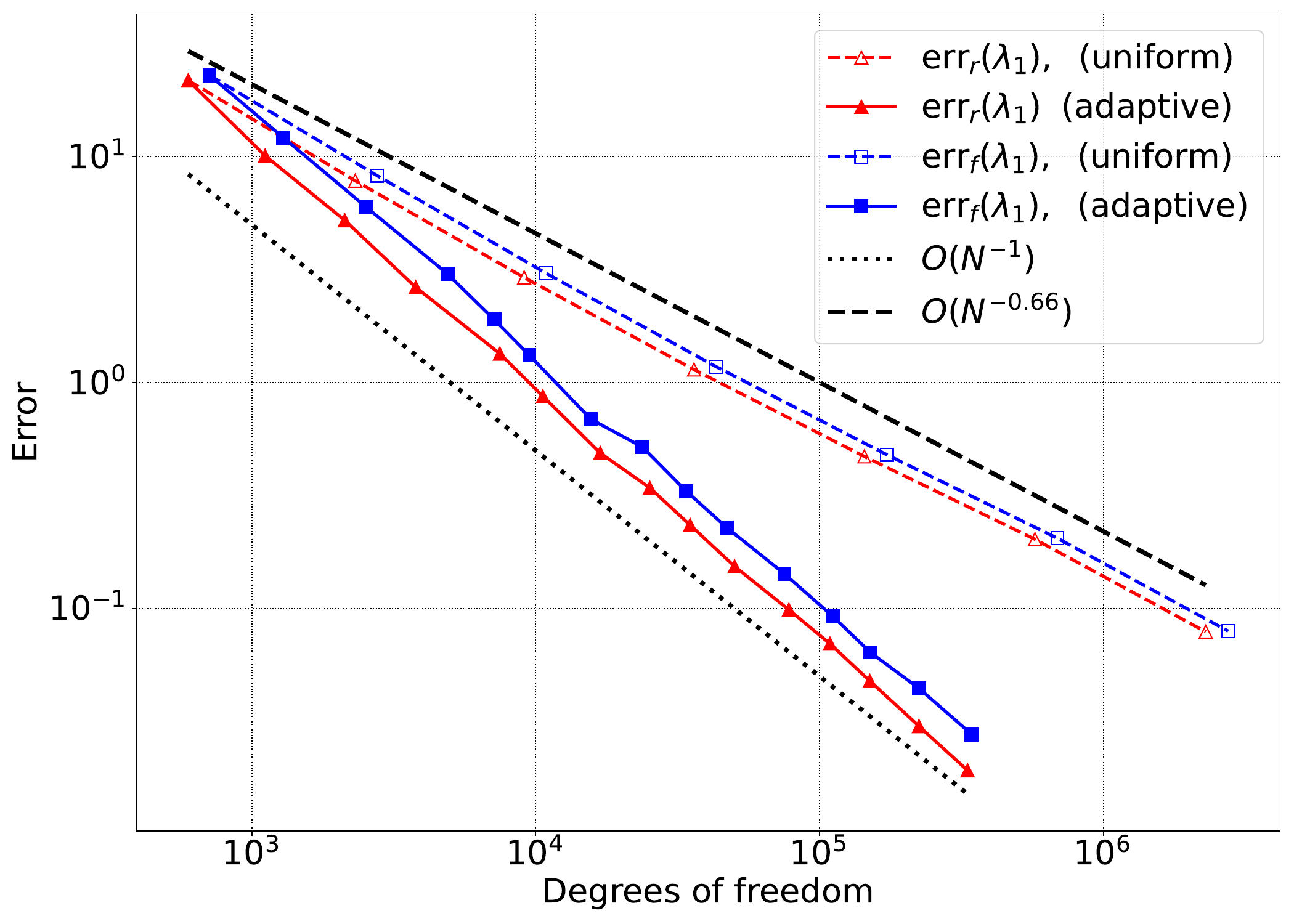}
	\caption{Test 2. Error curves for  $\theta$ and $\eta$  in the two dimensional T-shaped domain, compared with $\mathcal{O}(N^{-1})$.}
	\label{fig:tshape-error}
\end{figure}
\begin{figure}
	\centering
	\begin{minipage}{0.48\linewidth}
		\includegraphics[scale=0.06,trim= 35cm 2.5cm 35cm 2.5cm, clip]{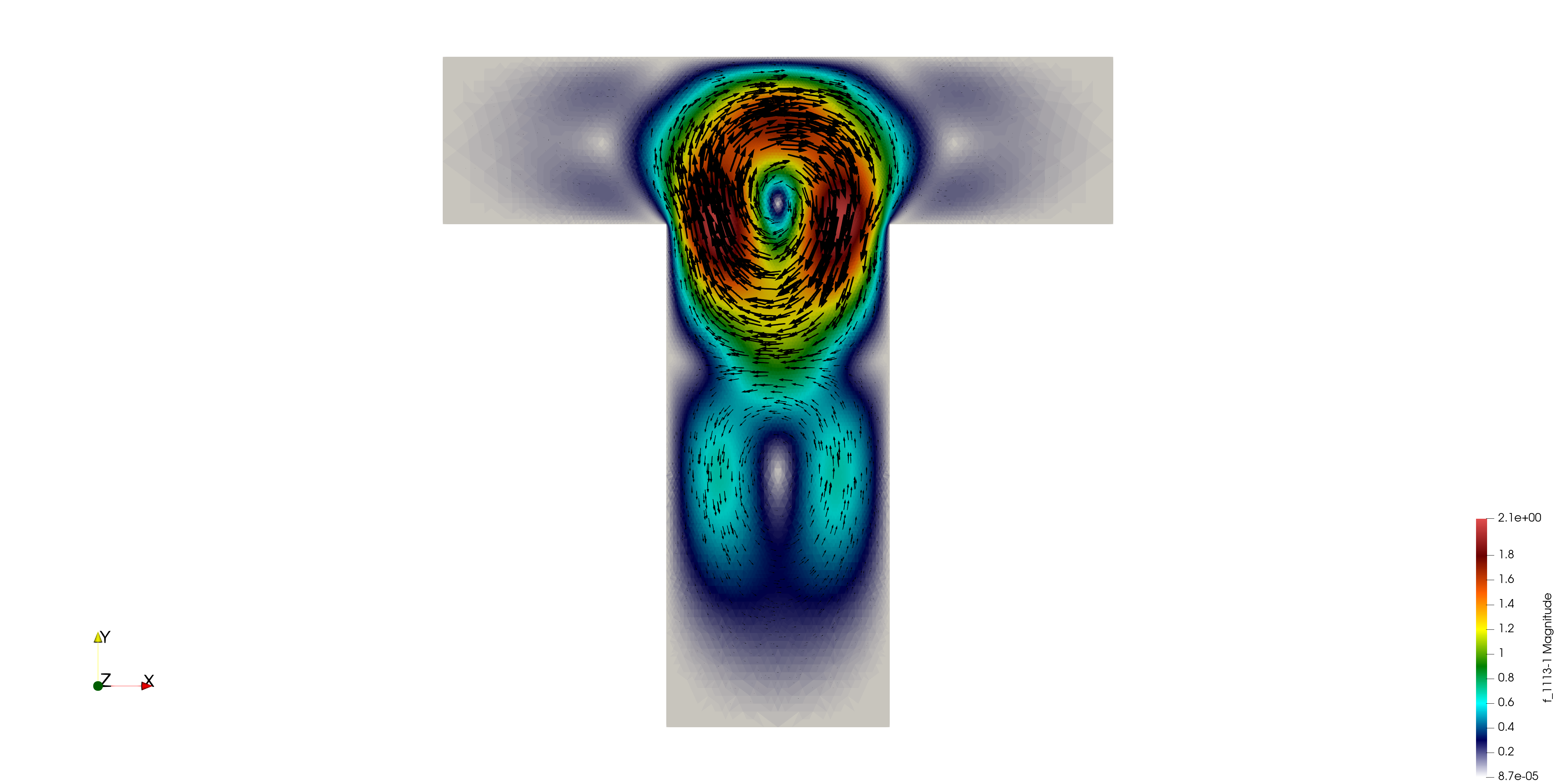}
	\end{minipage}
	\begin{minipage}{0.48\linewidth}
		\includegraphics[scale=0.06,trim= 35cm 2.5cm 35cm 2.5cm, clip]{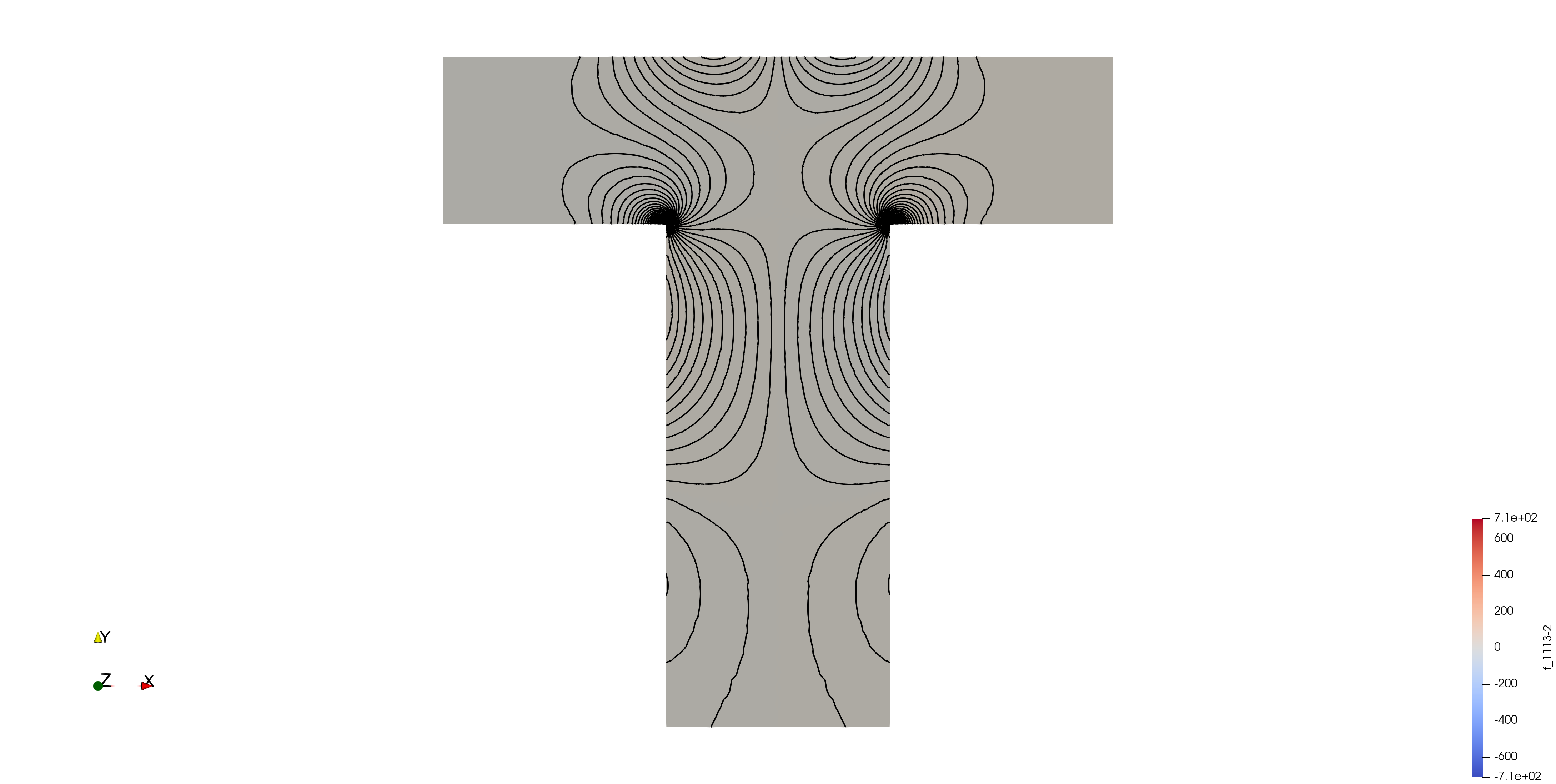}
	\end{minipage}
	\caption{Test 2. Lowest computed velocity field and pressure contour lines using the adaptive refinements and estimator $\eta$.}
	\label{fig:tshape-u-p}
\end{figure}

\subsubsection{Test 3. 3D L-shaped domain.} The goal of this test is to assess the performance of the numerical scheme when solving the eigenvalue problem in a three dimensional shape with a line singularity. The domain is an L shape given by
$$
\Omega:=(-1,1)\times(-1,1)\times(-1,0)\backslash\bigg((-1,0)\times(-1,0)\times(-1,0) \bigg).
$$
The domain presents a singularity on the line $(0,0,z)$, for $z\in[-1,0]$, whose initial mesh is depicted in Figure \ref{fig:L-shape-initial-mesh}. Hence, high gradients of the pressure, and consequently of the pseudostress are expected along this line (see Figure \ref{fig:tshape-3D-u-p}). In Table \ref{tabla:lshape-3D-uniform-vs-adaptive} we compare the performance of both numerical schemes. Both schemes shows that the adaptive scheme is capable to recover the optimal order of convergence $\mathcal{O}(N^{-2/3})$, where $\mathcal{O}(N^{-0.44})$ is the best order that we can expect when using uniform refinements. We remark that the computed convergence rate in this table has been obtained by excluding the first uniform and adaptive refinement. This is because the eigensolver has been configured with the shift as close as possible to the extrapolated eigenvalue. A different configuration or another eigensolver might yield a more accurate value for this first computation, without altering the trend shown here. For instance, in Figure \ref{fig:lshape3D-error} we show the error curves compared with the optimal convergence slope for each case.

In table \ref{tabla:lshape-3D-scheme1-vs-scheme2} we report the respective errors, estimators and effectivity indexes for each adaptive numerical scheme. We note that the estimators $\theta$ and $\eta$ behave like $\mathcal{O}(N^{-2/3})$, hence the effectivity indexes remain bounded above and below. This confirms numerically that the proposed estimators are reliable and efficient, as predicted by the theory.
On the other hand, we observe in Figure \ref{fig:lshape-3D-mesh-estimador-theta} some intermediate meshes obtained in the adaptive iteration using $\theta$ and $\eta$ estimators, respectively. 
We end the test by showing the computed velocity streamlines and pressure isosurfaces computed with the pseudostress-pressure-velocity model. Note the high pressure gradient along the line $(0,0,z)$.

\begin{figure}
	\centering
	\includegraphics[scale=0.06,trim= 25cm 3.5cm 25cm 3.5cm,, clip]{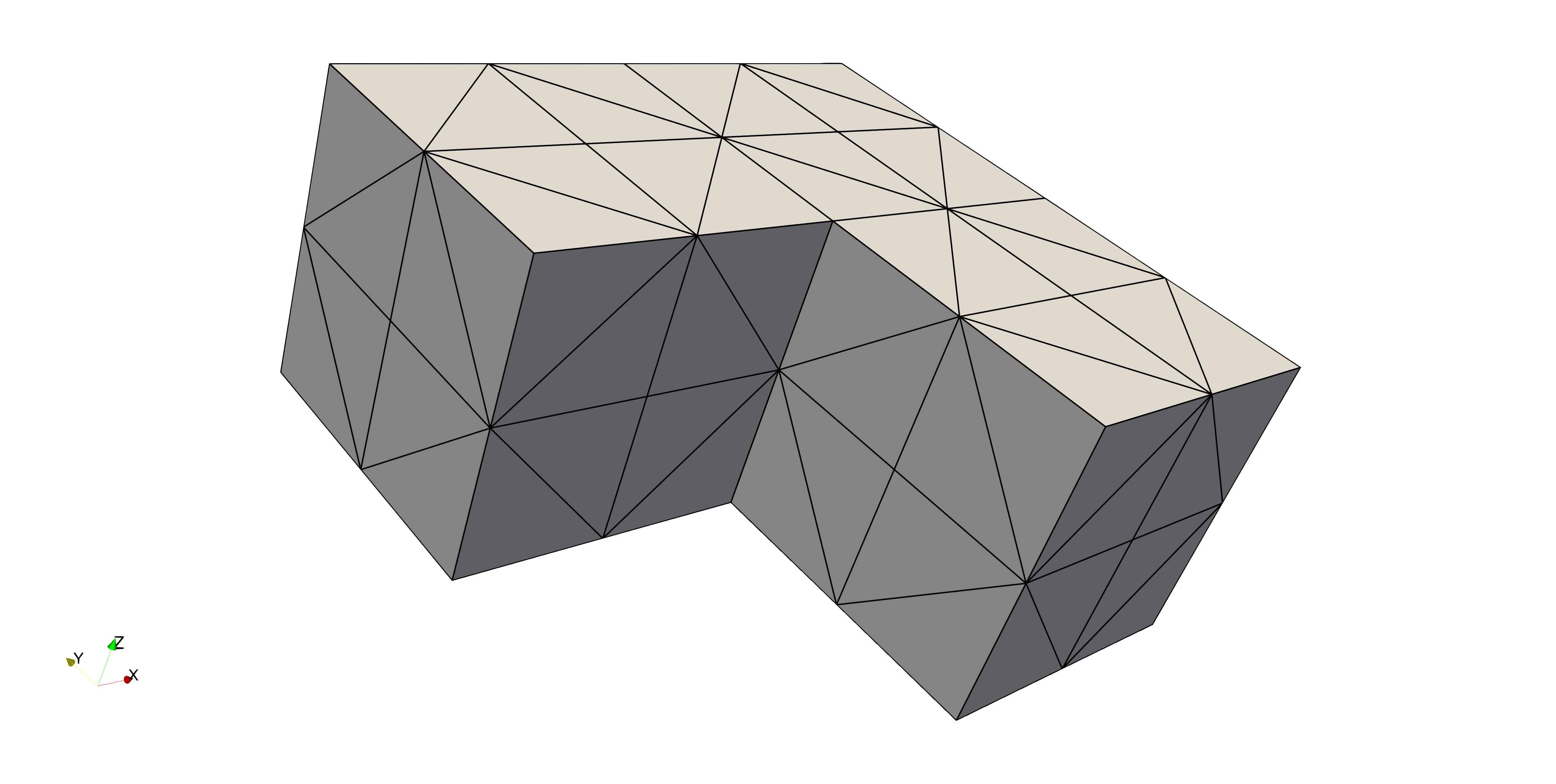}
	\caption{Test 3. Initial shape for the 3D L-shaped domain.}
	\label{fig:L-shape-initial-mesh}
\end{figure}
\begin{table}[h!]
	{\footnotesize
		\caption{Test 3: Comparison between the lowest computed eigenvalue in the pseudostress-velocity scheme using uniform and adaptive refinements.}
		\label{tabla:lshape-3D-uniform-vs-adaptive}
		\begin{center}
			\begin{tabular}{l c c c  |l c c c }
				\toprule
				\multicolumn{4}{c}{Pseudostress-velocity} & \multicolumn{4}{|c}{Pseudostress-pressure-velocity}\\
				\midrule
				\multicolumn{2}{c}{Uniform} &\multicolumn{2}{c}{Adaptive}& \multicolumn{2}{|c}{Uniform} & \multicolumn{2}{c}{Adaptive}\\
				\midrule
				$N$&$\lambda_{h1}$ & $N$&$\lambda_{h1}$&$N$&$\lambda_{h1}$ & $N$&$\lambda_{h1}$ \\ 
				\midrule
				2590	& 39.00981 &2590&		  39.00981 &2859		& 38.73581 &2859	& 38.73581  \\
				20041	& 38.93997 &15784		& 38.93228 &22193		& 38.84700 &17441	& 38.81631  \\
				157633	& 40.67234 &33490		& 40.08680 &174849		& 40.64580 &39910	& 40.09623 \\
				1250305	& 41.36395 &96043		& 40.89823 &1388033		& 41.35700 &109239	& 40.82324 \\
						& 		   &200857		& 41.24361 &			&  		   &237362	& 41.22484 \\
						&  		   &637204		& 41.54121 &			& 		   &697822	& 41.51951  \\
						&  		   &1216459		& 41.64638 &			&		   &1447421	& 41.64402 \\
						&  		   &3623905	    & 41.74001 &			&		   &4224335	& 41.73427 \\
				\midrule
				Order	&$\mathcal{O}(N^{-0.45})$		&Order	&$\mathcal{O}(N^{-0.66})$& Order	&$\mathcal{O}(N^{-0.45})$		&Order	&$\mathcal{O}(N^{-0.65})$ \\
				$\lambda_1$	& 41.81676 		&$\lambda_1$	& 41.81676 & $\lambda_1$	& 41.81676 		&$\lambda_1$	& 41.81676 \\
				\bottomrule             
			\end{tabular}
	\end{center}}
\end{table}
\begin{table}[h!]
	{\footnotesize
		\caption{Test 3: Computed errors, estimators and effectivity indexes on the adaptively refinement meshes for each numerical scheme.}
		\label{tabla:lshape-3D-scheme1-vs-scheme2}
		\begin{center}
			\begin{tabular}{c c c c|c c c c }
				\toprule
				\multicolumn{3}{c}{Pseudostress-velocity} &&& \multicolumn{3}{c}{Pseudostress-pressure-velocity}\\
				\midrule
				$\err_r(\lambda_{h1})$&$\theta^2$&$\eff_r(\lambda_{h1})$ &&& $\err_f(\lambda_{h1})$&$\eta^2$&$\eff_f(\lambda_{h1})$ \\ 
				\midrule
				2.80695e+00 &2.05543e+02& 1.36563e-02 &&&3.08095e+00	&2.13499e+02& 1.44308e-02  \\
				2.88448e+00	&7.57004e+01& 3.81039e-02 &&&3.00045e+00	&8.02030e+01& 3.74107e-02  \\
				1.72997e+00	&5.14925e+01& 3.35965e-02 &&&1.72053e+00	&5.15618e+01& 3.33684e-02 \\
				9.18533e-01	&2.70877e+01& 3.39096e-02 &&&9.93521e-01	&2.87218e+01& 3.45912e-02  \\
				5.73151e-01	&1.85319e+01& 3.09278e-02 &&&5.91919e-01	&1.89016e+01& 3.13158e-02  \\
				2.75556e-01	&8.73397e+00& 3.15499e-02 &&&2.97254e-01	&9.37408e+00& 3.17102e-02  \\
				1.70383e-01	&6.08206e+00& 2.80141e-02 &&&1.72741e-01	&6.22429e+00& 2.77527e-02 \\
				7.67489e-02 &3.00102e+00& 2.55742e-02 &&&8.24883e-02	&3.08255e+00& 2.67597e-02 \\
				\bottomrule             
			\end{tabular}
	\end{center}}
\end{table}
\begin{figure}
	\centering
	\begin{minipage}{0.49\linewidth}\centering
		\includegraphics[scale=0.06,trim= 25cm 4cm 25cm 4cm,, clip]{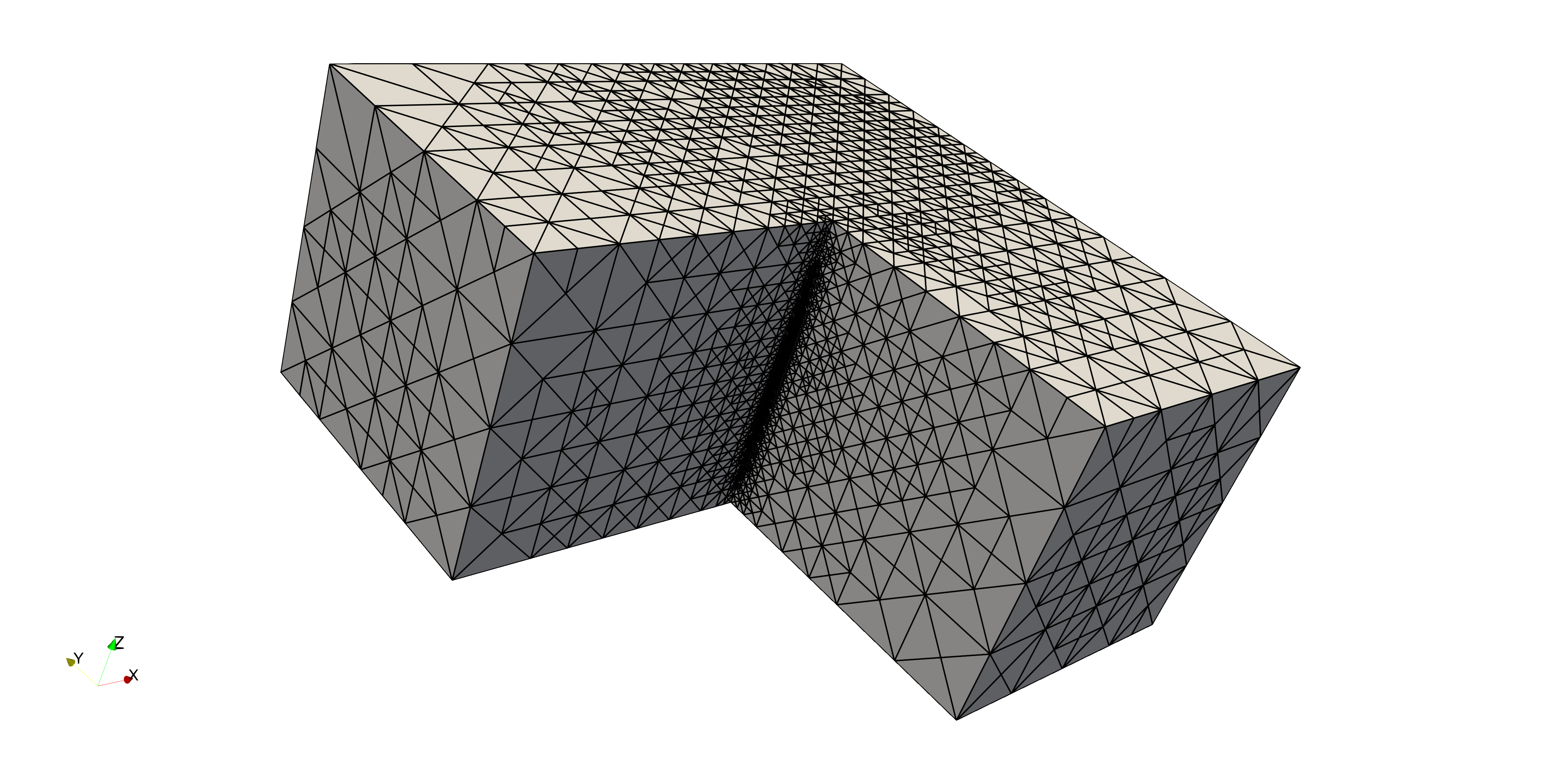}
	\end{minipage}
	\begin{minipage}{0.49\linewidth}\centering
		\includegraphics[scale=0.06,trim= 25cm 4cm 25cm 4cm,, clip]{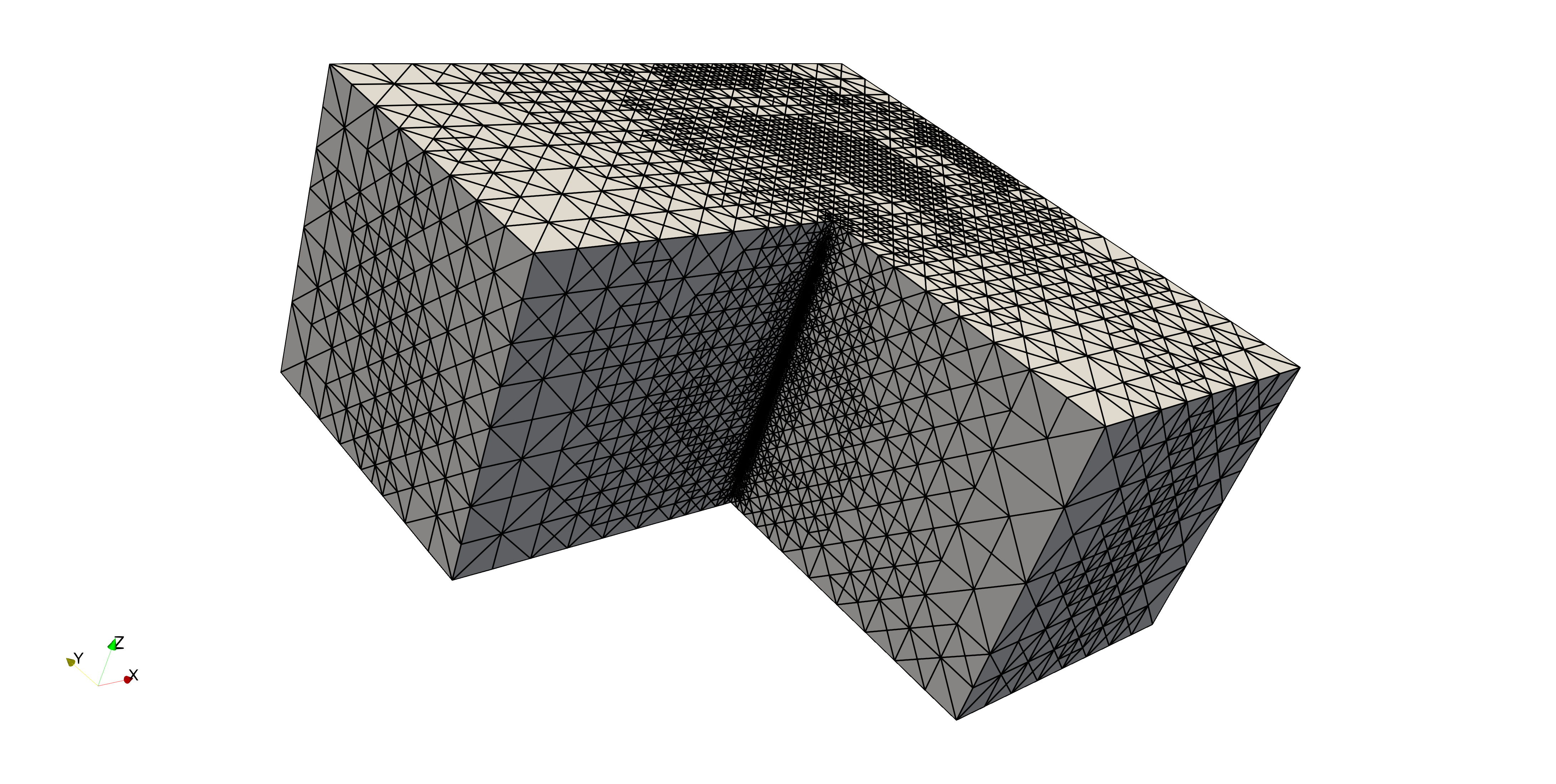}
	\end{minipage}\\
	\begin{minipage}{0.49\linewidth}\centering
		\includegraphics[scale=0.06,trim= 25cm 4cm 25cm 4cm,, clip]{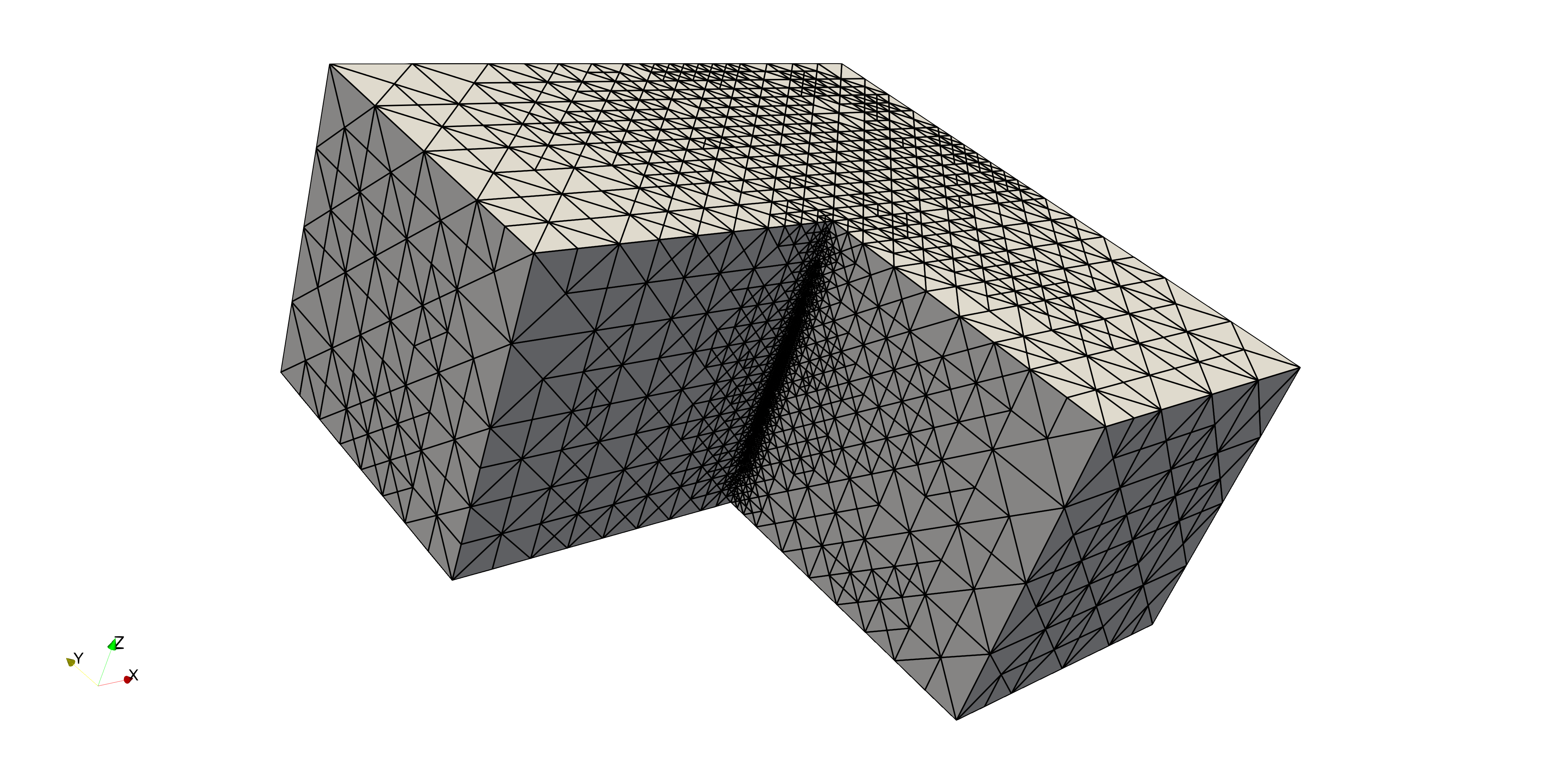}
	\end{minipage}
	\begin{minipage}{0.49\linewidth}\centering
		\includegraphics[scale=0.06,trim= 25cm 4cm 25cm 4cm,, clip]{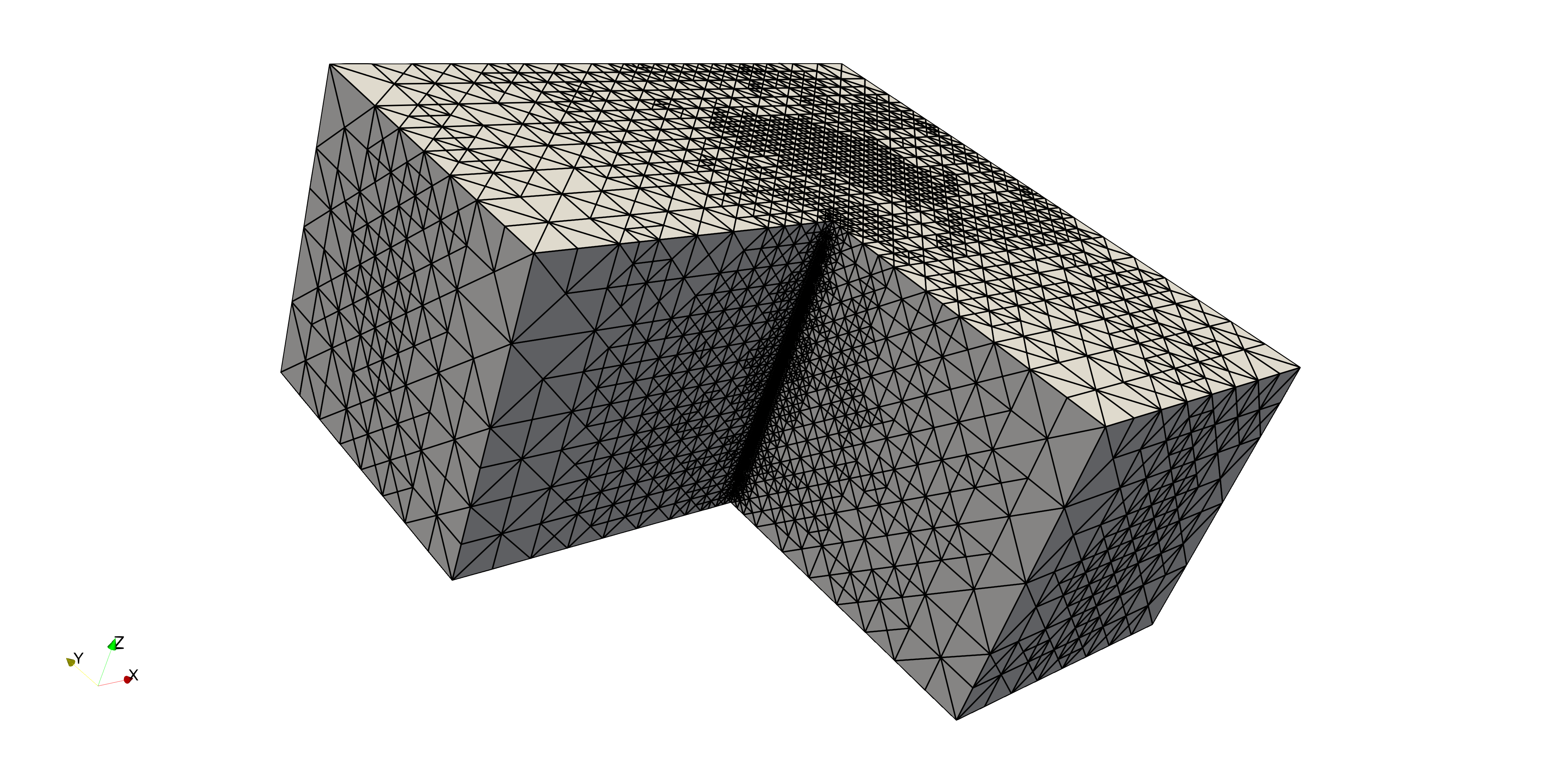}
	\end{minipage}\\
	\caption{Test 3. Intermediate adaptive meshes. Top row: meshes with $1216459$ and $3623905$ degrees of freedom using the estimator $\theta$. Bottom row: meshes with  $1447421$ and $4224335$ degrees of freedom using the estimator $\eta$. }
	\label{fig:lshape-3D-mesh-estimador-theta}
\end{figure}
\begin{figure}
	\centering
	\includegraphics[scale=0.25]{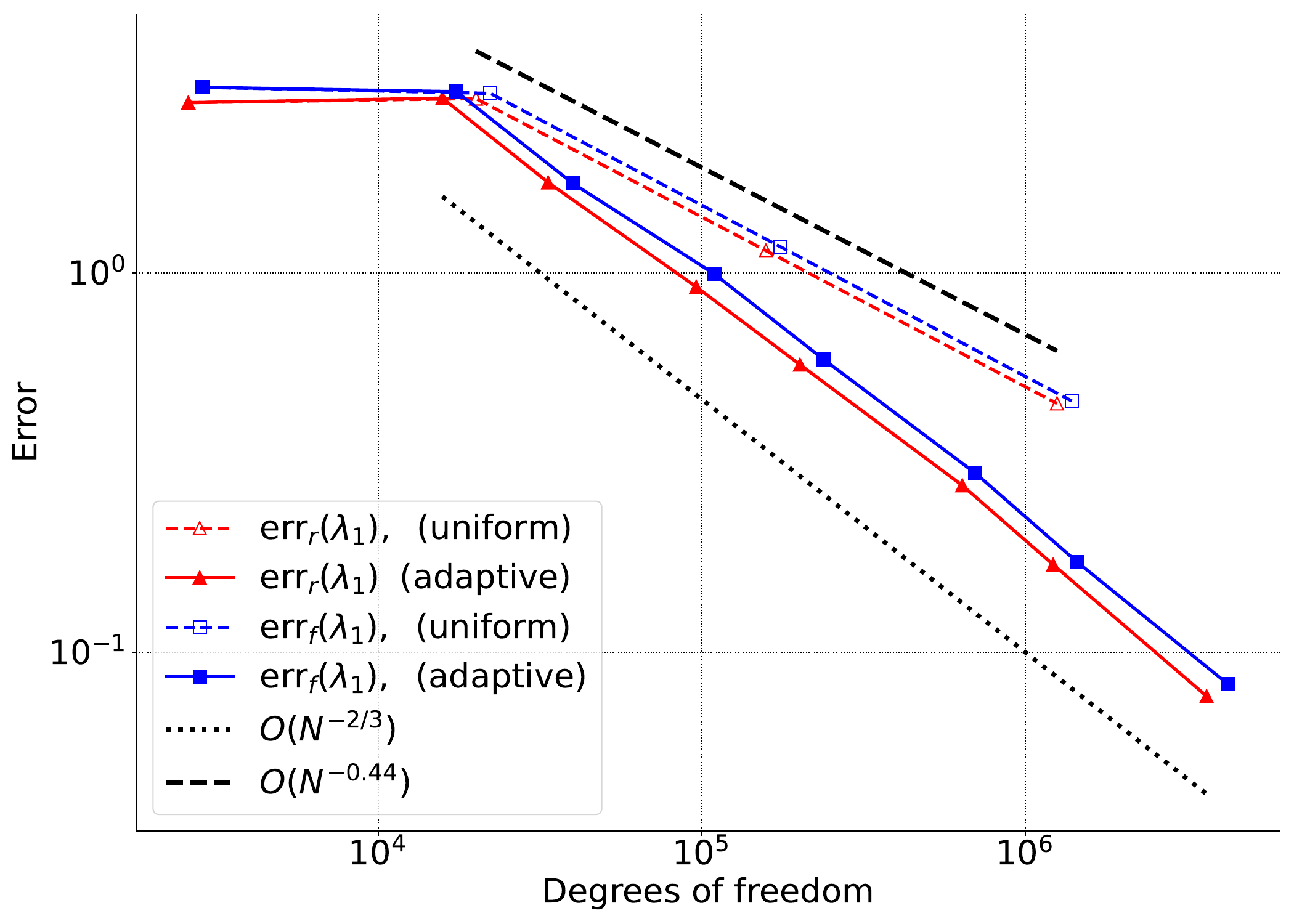}
	\caption{Test 3. Error curves for $\theta$ and $\eta$ in the three dimensional L-shaped domain compared with $\mathcal{O}(N^{-0.44})$ and $\mathcal{O}(N^{-2/3})$.}
	\label{fig:lshape3D-error}
\end{figure}
\begin{figure}
	\centering
	\begin{minipage}{0.49\linewidth}\centering
		\includegraphics[scale=0.06,trim= 25cm 4cm 25cm 4cm,, clip]{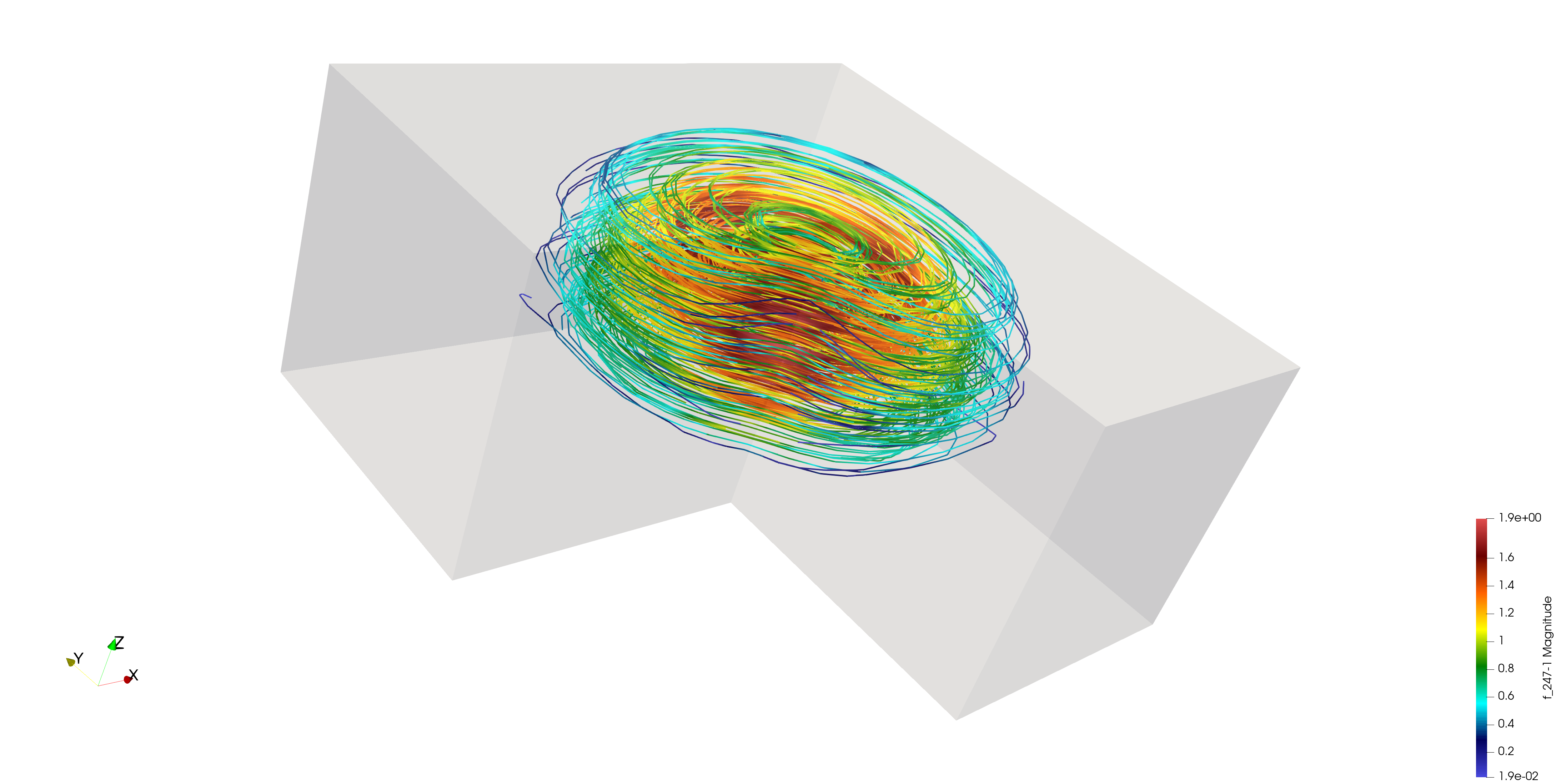}
	\end{minipage}
	\begin{minipage}{0.49\linewidth}\centering
		\includegraphics[scale=0.06,trim= 25cm 4cm 25cm 4cm,, clip]{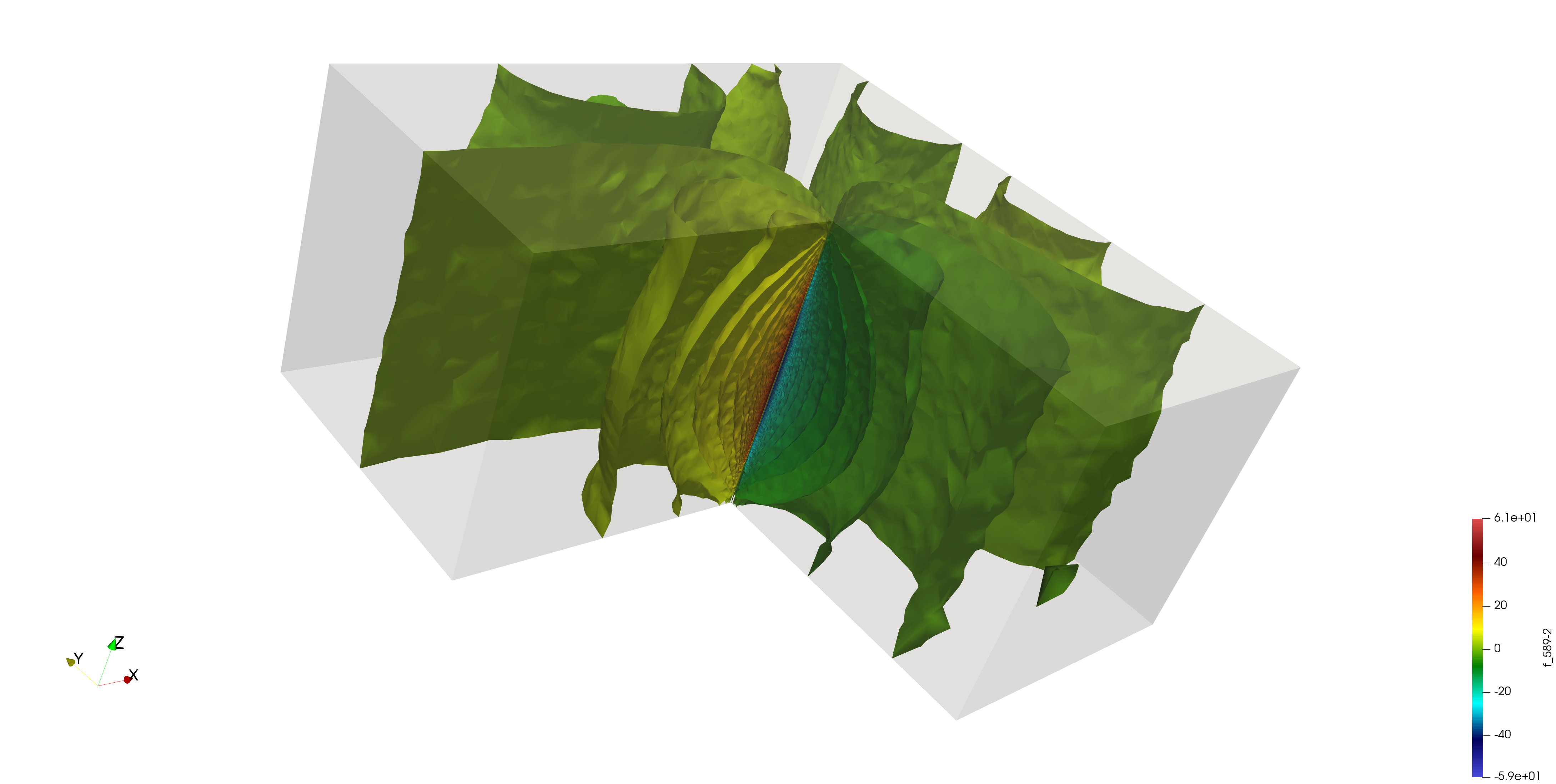}
	\end{minipage}
	\caption{Test 3. Computed velocity field streamlines and pressure isosurfaces using adaptive meshes and the $\eta$ estimator.}
	\label{fig:tshape-3D-u-p}
\end{figure}
\section{Compliance with Ethical Standards}
The authors have no conflicts of interest to declare that are relevant to the content of this article.
\bibliographystyle{siam}
\footnotesize
\bibliography{bib_LOQ}

\end{document}